\documentclass[12pt, reqno]{amsart}

\usepackage[utf8]{inputenc}
\usepackage[T1]{fontenc}
\usepackage{tikz,tikz-3dplot}
\usepackage{amsmath,amsfonts,amssymb,amsmath,latexsym,mathrsfs}
\usepackage{shuffle}
\usepackage[all]{xy}
\usepackage{stackengine}
 
\usepackage{enumerate}

\usepackage{hyperref}
\usepackage{tikz-cd} 

\usepackage{comment}
\usepackage{url}
\usepackage{tikz}

\usepackage{xfrac}

\newtheorem{theorem}{Theorem}[section]

\newtheorem{proposition}[theorem]{Proposition}
\newtheorem{definition}[theorem]{Definition}

\newtheorem{corollary}[theorem]{Corollary}

\newtheorem{remark}[theorem]{Remark}

\newtheorem{lemma}[theorem]{Lemma} 

\newtheorem{example}[theorem]{Example} 

\newcommand{\Z}{ {\mathbb Z} }
\newcommand{\C}{ {\mathbb C} }

\begin{document}

\title{On quandle linear representations}

\author{Mohamad \textsc{Maassarani} }
\maketitle
\begin{abstract}We consider complex representations of finite quandles. One can divide the paper into four parts. In the first part, we study complete reducibility of quandle linear representations, degree of irreducibles, existence of a faithful representation or not. In the second part, we relate irreducible quandle linear representations to projective representations of the inner automorphism group of the quandle. We prove that the irreducible representations of a quandle can be constructed using what we call quandle characters and the irreducible group representations of a given finite group that we call a representation group of the quandle. We also prove that irreducible representations of a finite quandle can be constructed up to equivalence algorithmically. In the third part, we show that the irreducible linear representations of the conjugacy quandle $Conj(G)$ of a (finite) group $G$, are all of the form a product of quandle character by an irreducible group representation of $G$ if and only if the Bogomolov multiplier of $G$ is trivial. It is the case for instance of many classical groups and simple groups. The study of irreducibles allows us to determine the enveloping group of $Conj(G)$ for $G$ a perfect group with trivial Bogomolov multiplier. In the last part, we show that two non isomorphic finite quandles can have equivalent Tannakian categories of finite dimensional representations over $\mathbb{C}$, which is not possible for finite groups .
\end{abstract}
\section*{Introduction and main results}
 A quandle is a set $Q$ equipped with a binary operation $\triangleright :Q\times Q\to Q$ satisfying some axioms. They appeared independently in \cite{DJ} and \cite{Ma} as invariant of knots. The fundamental quandle of a knot is a complete invariant of oriented knots up to mirror image. Quandles also appear for instance in : singularity theory \cite{Br}, classification of hopf algebras \cite{AG}, set-theoretic solutions of the Yang-Baxter equation (\cite{lebVen}, \cite{JP}). There is in fact a vast literature on the algebraic study of quandles and quandle homology.\\\\
Quandles can for instance be constructed from groups. For a group $G$, the conjugacy quandle $Conj(G)$ of $G$ is the set $G$ equipped with the binary operation $g\triangleright h=ghg^{-1}$. This allows to define a notion of quandle linear representation as introduced in \cite{rep} : a quandle representation is a quandle morphism $\rho$ from a quandle $Q$ to $Conj(GL(V))$  for some vector space $V$. Meaning $\rho(x\triangleright y)=\rho(x)\rho(y)\rho(x)^{-1}$. Group representations are examples of quandle representations of conjugacy quandles. A notion of strong representation is also introduced in \cite{rep}. In \cite{UE}, irreducible modules over what is known as the enveloping group (also known as associated group or structure group) of some finite quandles are considered. In fact, there is a correspondence between representations of the enveloping group and quandle representations. In \cite{str}, a relation between strong irreducible representations of a finite connected quandle and irreducibles of what is known as the finite enveloping group is given. Here we consider quandle representations over $\C$ of finite quandles in general (not necessarily strong). The key findings are the following :
 \begin{itemize}
\item[1)] Complete irreducibility criterion for representations of finite quandles.
\item[2)] The degree of an irreducible quandle representation of a finite quandle is bounded by $\sqrt{\vert Inn(Q) \vert}$ and divides $\vert Inn(Q) \vert$, where $Inn(Q)$ is the inner automorphism group of the quandle.
\item[3)] A way to construct the irreducible representations of some families of quandles.
\item[4)] The irreducible representations of finite quandles can be obtained from quandle characters and irreducible group representations of a finite group, what we call a representation group. 
\item[5)] Irreducible representations of finite quandles can be constructed up to equivalence algorithmically from the quandle. The algorithm can be implemented in the software  "Groups, Algorithms, and Programming" (GAP).
\item[6)] Reducing irreducible representations for conjugacy quandles of finite groups with trivial Bogomolov multiplier (as many classical groups and simple groups) to the study of the group linear irreducible representations and quandle characters which are easy to determine. 
\item[7)] Determining the enveloping group of conjugacy quandles of perfect groups with trivial Bogomolov multiplier (as simple groups).
\item[8)] Examples of non isomorphic finite quandles sharing equivalent Tannakian categories of finite dimensional representations over $\C$ , which is not possible for finite groups.
 \end{itemize}
\subsection*{Outline of the paper :}\:\\
In section $1)$, we recall definitions and collect properties from the literature related to quandles, inner automorphism groups of quandles, enveloping groups of quandles and linear representations of quandles. Section $2$, contains reminders on projective representation theory.\\\\
Section $3$, contains $1)$ and $2)$, as well as a proof of that for finite quandles : a quandle admits a finite dimensional faithful representation if and only if the quandle is injective, the irreducibles of an injective quandle separate the quandle. We also classify (part of $3)$) up to isomorphism, in section $3$, all the irreducible representation of the quandles $Q_{n,m}=\{x_i \vert i \in \mathbb{Z}/n\mathbb{Z} \} \cup \{y_i \vert i \in \mathbb{Z}/m\mathbb{Z} \}$ with binary operation :
$$x_i \triangleright y_j=y_{j+1},\quad y_i\triangleright x_j=x_{j+1} ,\quad x_{i}\triangleright x_j=x_j,\quad y_i\triangleright y_j=y_j.$$\\
In section $4$, we relate irreducible representations of a finite quandle $Q$ to irreducible projective representations of the inner automorphism group $Inn(Q)$ of $Q$. We show that if the Schur multiplier of $Inn(Q)$ is trivial then irreducibles of $Q$ can be constructed out of irreducibles of $Inn(Q)$ and quandle characters. We give some examples where this criterion apply. We provied a caracterisation of irreducibles of $Conj(G)$ for $G$ a Schur cover. We also prove $4)$ and $5)$, in section $4$.\\\\
In section $5$, we recall the definition of the Bogomolov multiplier $B_0(G)$ of a finite group $G$ and show that the subgroup of the Schur multiplier consisting of classes of symmetric $2$-cocycles is isomorphic to the Bogomolov multiplier. This is used in section $6$. \\\\
In section $6$, we relate irreducibles of conjugacy quandles to irreducible projective representations of the corresponding group. We show that the irreducible representations of $Conj(G)$ for $G$ finite are all obtained by taking the product of a quandle character by an irreducible linear representation of the group $G$ if and only if the $B_0(G)$ is trivial. This covers $6)$. The study of irreducibles allows using other ingredients to show that for a perfect group $G$ with $B_0(G)=0$, the enveloping group is isomorphic to $G\times \Z^{c(G)}$, where $c(G)$ is the number of conjugacy class of $G$. This covers $7)$. \\\\
In section $7$, we give three types of examples of two non isomorphic quandles sharing equivalent Tannakian categories of finite dimensional representations over $\C$. This covers $8)$. The first type of example involves a non injective quandle. The second, type concerns two non isomorphic injective quandles that are dual. The third type consists of two non isomorphic injective quandles comming from different conjugacy classes of some type of finite groups.

\section{Quandles, inner automorphisms, conjugacy quandles, enveloping groups and quandle representations}
We will recall definitions and collect properties from the literature related to : quandles, inner automorphism groups of quandles, enveloping groups of quandles and linear representations of quandles.

\subsection{Quandles, quandle morphisms}
\begin{definition}
 A quandle is a set $Q$ equipped with a binary operation $\triangleright:Q\times Q \to Q$  satisfying the following :
\begin{itemize}
\item[-] $x\triangleright x=x$ for $x\in Q$.
\item[-] $x \triangleright( y\triangleright z)= (x \triangleright y) \triangleright (x \triangleright z)$, for $x,y,z \in Q$. 
\item[-] For all $x,y \in Q$ there exists a unique $z\in Q$, such that $x \triangleright z=y$.
\end{itemize}
\end{definition}
\begin{definition}
 A quandle morphism is a map between two quandles $(Q,\triangleright_Q)$ and $(Q' ,\triangleright_{Q'})$ such that $f(z \triangleright_Q w)=f(z) \triangleright_{Q'} f(w)$, for $z,w \in Q$.
\end{definition}
The composition of two quandle morphisms is a quandle morphism.
\begin{definition}
A subquandle of $(Q,\triangleright)$ is a subset $Q'$ such that $\triangleright$ restricts to quandle structure on $Q'$.

\end{definition} The image of a quandle morphism is a subquandle.  
\subsection{Inner automorphisms of quandles}
For $Q$ a quandle it follows from the definition of a quandle we took that for $x\in Q$  the map $L_x: Q\to Q$ given by $L_x(y)=x\triangleright y$ is a bijective quandle morphism. 
\begin{definition}
For $Q$ a quandle, the group of inner automorphism $Inn(Q)$  of $Q$  is the subgroup of the group of bijections of $Q$ generated by the maps $L_x$ for $x\in Q$.
\end{definition}
\begin{proposition}
The map $Q\to Inn(Q), x\mapsto L_x$ is a quandle morphism if $Inn(Q)$ is endowed with the conjugacy quandle structure.
\end{proposition}
\begin{definition}
The orbits of a quandle $Q$ are the orbits of $Q$ under the action of $Inn(Q)$.
\end{definition}
\subsection{Conjugacy quandle of a group}
\begin{definition}
For a group $G$, the conjugacy quandle $Conj(G)$ of $G$ is a quandle consisting of the set $G$ equipped with the binary operation $x\triangleright y= xyx^{-1}$  for $x,y\in G$.
\end{definition}
A group morphism $f:G\to H$  is a quandle morphism with respect to the conjugacy quandle structure just defined. 
\begin{proposition}
For $G$ a group, the group $Inn(Conj(G))$ is equal to $Inn(G)\simeq G/Z(G)$ where $Z(G)$ is the center of $G$ and the orbits of $Inn(Conj(G))$ are the conjugacy classes of $G$.
\end{proposition}
\begin{proposition}
Any union of conjugacy classes of a group $G$ form a subquandle of $Conj(G)$.
\end{proposition}
 
\subsection{Envelopping group}\label{S14}

\begin{definition}
The enveloping group $G(Q)$ of a quandle $Q$ is the group defined by generators and relations :
$$G(Q)=\langle e_x, \:x \in Q \vert e_xe_ye_x^{-1}=e_{x\triangleright y} , \text{ for } x,y\in Q\rangle.$$
\end{definition}
The enveloping group is also called : the associated group, adjoint group or structure group. The group $G(Q)$ is infinite since word length (length of a generator is $1$ and the inverse of such element has length $-1$) induces a non trivial morphism $G(Q)\to \mathbb{Z}$. 
\begin{remark}
Let $O$ be an orbit of a quandle $Q$ the set $\{e_x\}_{x\in O}$ is an conjugacy class in $G(Q)$.
\end{remark}
\begin{proposition}
 \item[1)] The map $\varphi_Q : Q \to G(Q)$ mapping $x\in Q$ to the corresponding generator is a quandle morphism if we endow $G(Q)$ with the conjugacy quandle structure.
\item[2)]If $f:Q\to Conj(G)$ is a quandle morphism then there is a unique group morphism $f_{G(Q)}:G(Q)\to G$ such that the following diagram commutes :
$$\begin{tikzcd}
G(Q) \arrow{r}{f_{G(Q)}} &G  \\ 
Q \arrow{u}{\varphi_Q} \arrow{ru}{f} & 
\end{tikzcd} $$
 \end{proposition}
\begin{definition}
A quandle $Q$ is called injective if $\varphi_Q : Q \to G(Q)$ is injective.
\end{definition}
\begin{proposition}
A subquandle of $Conj(G)$ is injective.
\end{proposition}
The above proposition follows by applying the universal property of the enveloping group to the inclusion of the subquandle into $Conj(G)$. \\\\
We have seen that $x\mapsto L_x$ is a quandle morphism $\theta$ from $Q$ to $Conj(Inn(Q))$, hence we have a unique group morphism $\theta_Q : G(Q) \to Inn(Q)$ such that the following diagram commute : 
$$\begin{tikzcd}
G(Q) \arrow{r}{\theta_Q} &G  \\ 
Q \arrow{u}{\varphi_Q} \arrow{ru}{\theta} & 
\end{tikzcd} $$
\begin{proposition}[\cite{EM}]\label{cent}
The kernel of $\theta_Q : G(Q) \to Inn(Q)$ lies in the center of $G(Q)$. In particular, if $Q$ is finite, the center of $G(Q)$ has finite index in $G(Q)$.
\end{proposition}
Since a subgroup of finite index of a finitely generated group is finitely generated, we have :
\begin{corollary}\label{corcent}
The center of $G(Q)$ and the kernel of $\theta_Q$ for $Q$ a finite quandle are finitely generated abelian groups.
\end{corollary}
\begin{proposition}[\cite{dar}]\label{centinj}
If $Q$ is an injective quandle, then the kernel of $\theta_Q :G(Q) \to Inn(Q)$ is the center of $G(Q)$.
\end{proposition}

The following is well known. We sketch part of the proof for completeness.
\begin{proposition}\label{abder}
Let $Q$ be a quandle.
\item[1)] The abelianisation $G(Q)^{ab}$ of $G(Q)$ is isomorphic to the free abelian group on the set of orbits of $Q$ and the abelianisation maps can be identified to the morphism mapping $e_x$ to its orbit for $x\in Q$.
\item[2)] If $Q$ is finite then the derived subgroup $G(Q)'$ of $G(Q)$ is finite. It consists of the torsion elements of $G(Q)$.
\end{proposition}
\begin{proof}
$1)$ can be proved by manipulating the presentation of $G(Q)$. Since $G(Q)^{ab}$ is free abelian all torsion elements of $G(Q)$ belong to $G(Q)'$. Hence, to prove $2)$ we only need to prove that $G(Q)'$ is finite for $Q$ finite. But this true since Schur theorem for the derived group states that if the center has finite index (see proposition \ref{cent}) then the derived group is finite. 
\end{proof}
\subsection{Quandle linear representation}
\begin{definition}
\item[1)] A linear representation of a quandle $Q$ over a vector space $V$ is a quandle morphism $\rho: Q \to Conj(\mathrm{GL}(V))$, i.e. : 
$$\rho(x\triangleright y)=\rho(x) \rho(y) \rho(x)^{-1},$$
for all $x,y \in Q$.
\item[2)] A subrepresentation of $\rho$ is a vector subspace $W\subset V$, such that $\rho(x)(W) \subset W$ for all $x\in Q$.
\item[3)] An irreducible representation is a representation $V$ that has no subrepresentations other than $0$ and $V$.
\item[4)] For $Q$ a quandle, the degree of a quandle linear representation $\rho : Q \to GL(V) $ is the dimension of $V$
\end{definition} 
We define intertwining operators and equivalence of quandle representations as for groups. Note that if $f:Q \to Conj(G)$ is a quandle morphism and $\rho : G \to GL(V)$ is a group represenation of $G$, then $\rho\circ f$ is a quandle representation. In particular, a group linear representation is a quandle linear representation of the conjugacy quandle of the group.\\\\
By the universal property of the enveloping group, any linear represantation $\rho : Q\to GL(V)$ of a quandle $Q$ extends to a group representation $\rho_{G(Q)} : G(Q) \to GL(V)$ such that  the following diagram commute :

$$\begin{tikzcd}
G(Q) \arrow{r}{\rho_{G(Q)}} &G  \\ 
Q \arrow{u}{\varphi_Q} \arrow{ru}{\rho} & 
\end{tikzcd} $$
 
 As noticed in \cite{str}, we have hence a correspondance between linear representations of $Q$ and linear representations of the group $G(Q)$.
\begin{proposition}\label{correspondancerep}
 \item[1)] We have a one to one correspondance between quandle representations of $Q$ and group representations of $G(Q)$ given by $\rho \mapsto \rho_{G(Q)}$.
\item[2)] If $\rho$ is irreducible then $\rho_{G(Q)}$ is irreducible.
 \end{proposition}
\begin{proof}
$1)$ follows from the universal property and the fact that a group represenation of $G(Q)\to GL(V)$ composed by $\varphi_Q$ gives a quandle represenation $Q\to GL(V)$. $2)$ is true since $\rho_{G(Q)}\circ \varphi_Q=\rho$. Indeed, if $\rho$ has no proper subspace stable by the $\rho(x)$ for all $x\in Q$, $\rho_{G(Q)}$ has no propoer subspaces stable by $\rho_{G(Q)}(e_x)=\rho(x)$ for all $x\in Q$.
\end{proof}

\section{Reminders on projective representations}
In this section the groups we consider are finite unless we state the contrary and the vector spaces are complex. We recall some facts on projective representations. The content of this section is covered in \cite{Sc} for instance.
\begin{definition}
\item[1)]  A projective representation of a group $G$ is a morphism $G\to PGL(V)$.
\item[2)]  A projective representation $G \to PGL(V)$ is irreducible if $P(V)$ admit no proper projective subspace stable by $G$.
\item[3)] Two projective representations are projectively equivalent if they are conjugate one to another by an invertible projective transformation.
\end{definition} 
Let $\rho : G \to PGL(V)$ be a projective representation. Let $\tilde{\rho}$ be a lift of $\rho$ to $GL(V)$ such that $\tilde{\rho}(1)=Id$. For $g,h\in G$ :
$$\tilde{\rho}(gh)=\alpha_{\tilde{\rho}}(g,h)\tilde{\rho}(g)\tilde{\rho}(h),$$
for some $\alpha_{\tilde{\rho}}(g,h)\in \mathbb{C}^\times$. One can check that the mapping $\alpha_{\tilde{\rho}} :G\times G\to \mathbb{C}^\times$ corresponding to the constants is a $2$-cocycle with values in $\mathbb{C^\times}$. The map $\tilde{\rho}$ is called $\alpha_{\tilde{\rho}}$-projective representation. Now, let $\tilde{\rho}'$ be another lift of $\rho$ to $GL(V)$ such that $\tilde{\rho}'(1)=Id$. The cocycles $\alpha_{\tilde{\rho}}$ and $\alpha_{\tilde{\rho}'}$ differ by a boundary. Hence, the projective representation $\rho$ defines a cohomology class $[\rho]$ in the group $H^2(G,\mathbb{C}^\times)$ called the  Schur multiplier of $G$. The Schur multiplier of a finite group is finite.
\begin{proposition}
\item[1)] A projcetive representation $\rho :G \to PGL(V)$ lifts to a linear representation $G\to GL(V)$ if and only if its class in $H^2(G,\mathbb{C}^\times)$ is trivial.
\item[2)] Two projectively equivalent representations definie the same cohomology class in $H^2(G,\mathbb{C}^\times)$.
\end{proposition}
One can use twisted algebras to prove the following (\cite{Sc} p. 40-42) :
\begin{proposition}\label{classir}
One can associate to each class $\alpha$ of the Schur multiplier $H^2(G,\mathbb{C}^\times)$ at least one irreducible projective representation of $G$ over $\mathbb{C}$ having $\alpha$ as a class.
\end{proposition}
Let $\theta: G\to H$ be a group morphism.  To a $2$-cocyle $\alpha$ of $H$ with values in $\mathbb{C}^\times$ one can assign the $2$-cocycle $\alpha_{\theta}=\alpha \circ (\theta \times \theta)$ of $G$ with values in $\mathbb{C}^\times $. This assignement induces a morphism $H^2(H,\mathbb{C}^\times)\to H^2(G,\mathbb{C}^\times)$ called the inflation map and will be denoted by $Inf_\theta$. 
\begin{proposition}\label{kerinfrho}
Let $\rho:H\to PGL(V)$ be a projective representation and $\theta : G\to H$ be a group morphism. $\rho$ lifts to a linear representation of $G$ via $\theta$, $($i.e. $\rho \circ \theta$ lifts to a linear representation of $G$$)$ if and only if the class of $\rho$ in $H^2(H,\mathbb{C}^\times)$ lies in the kernel of the inflation map $Inf_\theta$. 
\end{proposition}
\begin{proposition}
Let $1\to A \to G \overset{\theta}{\to} H $ be a central extension ($G$ and $H$ need not to be finite). We have an exact sequence $($inflation-restriction exact sequence with coefficient in $\mathbb{C}^\times$$)$ :
$$1 \to  Hom(H,\mathbb{C}^\times)\to Hom(G,\mathbb{C}^\times)\overset{Res}{\to} Hom(A,\mathbb{C}^\times)\to H^2(H,\mathbb{C}^\times)\overset{Inf_\theta}{\to}H^2(G,\mathbb{C}^\times),$$
where $Res$ is the restriction morphism.
\end{proposition}
\begin{definition}
A central extension $1\to A \to G \overset{\theta}{\to} H $ is called a stem extension if the nucleus $A$ seen as a subgroup of $G$ lies in the derived subgroup of $G$.
\end{definition}
For a stem extension $1\to A \to G \overset{\theta}{\to} H $ the inflation-restriction exact sequence gives an exact sequence :
$$0 \to Hom(A,\mathbb{C}^\times)\to H^2(H,\mathbb{C}^\times)\overset{Inf_\theta}{\to}H^2(G,\mathbb{C}^\times).$$
Hence, $A$ is isomorphic to a subgroup of the Schur multiplier for a stem extension.
\begin{definition}
A group $G$ is a Schur cover of $H$, if we have a morphism $\theta :G \to H$, such that $1\to ker(\theta) \to G \overset{\theta}{\to} H\to 1$ is a stem extension with $ker(\theta) \simeq H^2(H,\mathbb{C}^\times)$.
\end{definition}
Hence, if $G$ is a Schur cover of $H$ associated to the morphism $\theta: G\to H$, the kernel of $Inf_\theta$ is the whole group $H^2(H,\mathbb{C}^\times)$. In particular : 
\begin{proposition}
If $G$ is a Schur cover of $H$ associated to the morphism $\theta: G\to H$, then any projective irreducible representation $\rho$ of $H$ lifts via $\theta$ to an irreducible linear representation of $G$ $($i.e. $\rho \circ \theta$ lifts to a linear representation of $G$$)$. 
\end{proposition}

\section{Some generalities on quandle representations}
In this section, we have $5$ subsections. All the results are for finite quandles and representations over $\C$. In subsection $1$, we show that a finite dimensional quandle representation is completely reducible if and only if the elements of $Q$ acts in the representation via diagonalizable elements. In the second subsection, we remark that finite dimensional unitary quandle representations are completely reducible, that an irreducible finite dimensional quandle representation is unitary for some inner product if and only if any element of $Q$ act via an elements of deteminant having modulus $1$ and that an irreducible finite dimensional quandle representation can be twisted by a quandle character to obtain a unitary irreducible representation. In subsection $3$, we show that an irreducible quandle representation is finite dimensional and that its degree divides $\vert Inn(Q) \vert$ and is bounded by $\sqrt{\vert Inn(Q) \vert}$, where $Inn(Q)$ is the inner automorphism group of the quandle. In subsection $4$, we mainly  prove that a quandle admits a faithful finite dimensional representation if and only if the quandle is injective. We also prove that for an injective quandle, the irreducibles seperate the quandle. In the last subsection, we construct the irreducible representations of a given family of quandles and classify them up to isomorphism.
\subsection{Complete reduciblity criterion}
\begin{definition}
A quandle representation will be called completely reducible if it can be decomposed as the direct sum of irreducible subrepresentations.
\end{definition}

As noticed in \cite{nondec}, in general a quandle representation is not always completely reducible. Indeed, if $Q$ is a quandle the mapping $\rho :Q \to GL(\mathbb{C}^2)$ which assigns to every element of $Q$ the endomorphism having the following matrice in the canonical basis :
\[\begin{pmatrix}
1 & 1\\
0 & 1
\end{pmatrix}\]
is a quandle representation. As one checks this representation is not completely reducible.\\\\
As we have seen in the first section, if $\rho : Q \to GL(V)$ is a quandle representation, then by the universal property of the map $\varphi_Q$ from $Q$ to the envelopping group $G(Q)$ there is a unique group representation $\rho_{G(Q)}: G(Q)\to GL(V)$ such that $\rho_{G(Q)}=\rho \circ \varphi_Q$.
\begin{proposition}\label{corred}
Let $\rho$ be a finite dimensional representation of $Q$.
\item[1)] A subrepresentation for $Q$ with respect to $\rho$ is a subrepresentations for $G(Q)$ with respect to $\rho_{G(Q)}$ and vice versa.
\item[2)] The quandle representation $\rho$ is completely reducible if and only if the corresponding group representation $\rho_{G(Q)}$ is completely reducible.
\end{proposition}
\begin{proof}
We leave the proof for the reader.
\end{proof}
\begin{corollary}
If $\rho$ is finite dimensional, then $\rho$ is irreducible if and only if $\rho_{G(Q)}$ is.
\end{corollary}

\begin{proposition}\label{Zp} For $Q$ a finite quandle let $Z_p$ be the subgroup of $G(Q)$ generated by the elements $e_x^n$ for $x\in Q$ and with $n=\vert Inn(Q) \vert $.
\item[1)] The group $Z_p$ lies in the center of $G(Q)$.
\item[2)] $Z_p$ has finite index in $G(Q)$.
\end{proposition}
\begin{proof}
Recall (see subsection \ref{S14}) that we have a group morphism $\theta_Q : G(Q) \to Inn(Q)$ such that $\theta_Q(e_x)=L_x$ for $x\in Q$. The element $e_x^n$ for $n$ as in the proposition lies in the kernel of $\theta_Q$ and hence is central by proposition \ref{cent}  of subsection \ref{S14}. This proves $1)$. It follows using $1)$ proposition \ref{abder} that the image of $Z_p$ has finite index in $G(Q)^{ab}$. By the same proposition the derived subgroup $G(Q)'$ of $G(Q)$ is finite. $2)$ is a consequence of the two last remarks. 
\end{proof}
We note that groups obtained from enveloping groups as quotient by chosen powers of the generators are considered in the litterature.
\begin{lemma}\label{power}
An automorphism $f$ of a finite dimensional vector space over $\mathbb{C}$ is diagonalizable if and only if $f^m$ is diagonalizable for some $m>0$.
\end{lemma}
\begin{proof}
If $f$ is diagonalizable then $f^m$ is diagonalizable. If $f^m$ is diagonalizable the minimal polynomial $P(X)$ of $f^m$ has simple roots different from zero. Therefore, the polynomial $R(X)=P(X^m)$ has simple roots. But, $R(f)=P(f^m)=0$ and $R$ has simple roots, hence $f$ is diagonalizable. This proves that if $f^m$ diagonalizable then $f$ is.
\end{proof}

\begin{proposition}\label{dec}
Let $1\to Z \to G \to H\to 1$ be a central extension of groups with $H$ finite and $\rho :G\to GL(V)$ be a finite dimensional complex group representation over $\C$.
\item[1)] The elements of $\rho(Z)$ are all diagonalizable if and only if all elements of $\rho(G)$ are diagonalizable.
\item[2)] The representation $\rho$ is completely reducible if and only if every element of $\rho(G)$ $($equivalently $\rho(Z)$$)$ is diagonalizable.
\end{proposition}
\begin{proof}
We prove $1)$. For $g\in G$, the element  $g^{\vert H \vert }$ lies in $Z$ and hence by the previous lemma $\rho(g)^{\vert H \vert }$ is diagonalizable if and only if $\rho(g)$ is. $1)$ follows.
We now prove $2)$. If $\rho$ is decomposable into sum of irreducibles $\bigoplus_i V_i$ then by Schur's lemma $Z$ acts by multiplication by scalars in each $V_i$ and hence any element of $\rho(Z)$ is diagonalizable. We have proved that if $\rho$ is decomposable into sum of irreducibles then $\rho(z)$ is diagonalizable for every $z\in Z$. We now prove the converse. Assume that all the elments $\rho(z)$ for $z\in Z$ are diagonalizable. This means that the elements of $\rho(Z)$ can be simultaneously diagonalized since they commute. Hence, there is a basis $(v_1,\dots,v_n)$ of $V$ such that for all $z\in Z$ : 
$$ \rho(z) v_i =\chi_i(z) v_i,$$ 
with $\chi_i :Z \to \mathbb{C}^\times$ a multiplicative character. Let $X$ be the set formed by the characters $\chi_1,\cdots,\chi_n$. For $\chi \in X$, let $V_\chi$ be the subspace of $V$ spanned by the vector $v_i$ such that $\chi_i=\chi$. The space $V_\chi$ is stable under the action of $G$ and :
$$ V= \bigoplus_{\chi\in X} V_{\chi}.$$
For $\chi \in X$, denote by $\rho_\chi : G\to GL(V_\chi)$ the representation obtained from $\rho$ by restriction to $V_\chi$ and by $\varphi_\chi$ the projection $GL(V_\chi)\to PGL(V_\chi)$. The subgroup $Z$ of $G$ acts in $V_\chi$ by multiplication by scalars and $G$ is the central extension of the finite group $H$ by $Z$. Hence, the image of $\varphi_\chi \circ \rho_\chi$ is finite. Since we are working over $\mathbb{C}$, the projection $\theta_\chi : SL(V_\chi)\to PGL(V_\chi)$ ($SL$ for endomorphisms of determinant $1$) is surjective and has finite kernel. Hence, the preimage $G_\chi$ of $\varphi_\chi \circ \rho_\chi(G)$ under $\theta_\chi$ is a finite group. Since $G_\chi$ is finite, $V_\chi$ decomposes as sum of irreducibles under the natural action of $G_\chi \subset SL(V_\chi)$ : 
$$V_\chi = \bigoplus_i W_i$$
 Now any element of $\rho_\chi(G)$ is obtained from an element of $G_\chi$ by a multiplication by a scalar and vice versa. Hence, any stable subspace of $V_\chi$ under $G_\chi$ is stable under $G$ and vice versa. Therefore, the decomposition of $V_\chi= \bigoplus_i W_i$ obtained above is a decomposition of $V_\chi$ into irreducibles with respect to the representation $\rho_\chi :G \to GL(V_\chi)$. This proves that $V$ decomposes into a direct sum of irreducible representations if $\rho(z)$ is diagonalizable for all $z\in Z$. We have completed the proof of $2)$ of the proposition.
\end{proof}

\begin{theorem}
Let $Q$ be a finite quandle and $\rho :Q \to GL(V)$ be a finite dimensional quandle representation over $\mathbb{C}$. The quandle representation $\rho$ is completely reducible if and only if all the elements of $\rho(Q)$ are diagonalizable.
\end{theorem}
\begin{proof}
By $2)$ of proposition \ref{corred}, $\rho$ is completely reducible if and only if the correspondant group representation  $\rho_{G(Q)}:G(Q)\to GL(V)$ is. From proposition \ref{Zp}, the subgroup $Z_p$ of $G(Q)$ generated by $e_x^n$ for $x\in Q$ and $n=\vert Inn(Q)\vert$ lies in the center of $G(Q)$ and has finite index. Applying $2)$ of the previous proposition to the central exact sequence $1\to Z_p\to G(Q) \to G(Q)/Z_p \to 1$, we get that $\rho_{G(Q)}$ is completely reducible if and only if the elements of $\rho_{G(Q)}(Z_p)$ are diagonalizable. But $Z_p$ is abelian and the diagonalizability of the image of its elements is equivalent to the diagonalizability of the image of its generators. Hence, $\rho_{G(Q)}$ is completely reducible if and only if the element $\rho_{G(Q)}(e_x)^n$ for $x\in Q$ with $n=\vert Inn(Q) \vert $ are diagonalizable. By lemma \ref{power}, $\rho_{G(Q)}(e_x)^n$ is diagonalizable if and only if $\rho_{G(Q)}(e_x)=\rho(x)$ is. This proves the theorem.
\end{proof}
\begin{corollary}
For $Q$ a finite quandle and $\rho : Q \to GL(V)$ a finite dimensional irreducible quandle representation over $\mathbb{C}$, the image of $\rho$ consists of diagonalizable automorphisms.
\end{corollary}
\begin{remark}
We note that the last theorem could fail if the quandle is not finite. Indeed, let $G$ be the subgroup of $GL_2(\mathbb{C})$ generated by the matrices :
\[\begin{pmatrix}
1 & 1\\
0 & 1
\end{pmatrix}\quad \text{and} \quad \begin{pmatrix}
1 & 0\\
1& 1
\end{pmatrix}.\]
$G$ is infinite since the matrices don't have a finite order. $Conj(G)$ is a quandle and the natural inclusion of $Conj(G)=G$ into $GL_2(\mathbb{C})$ is an irreducible quandle representation. Indeed, the only nonzero stable subspace shared by the two matrices is $\mathbb{C}^2$. But the matrices given are not diagonalizable. 
\end{remark}
\begin{example}
If $Q$ is a finite quandle, $G$ is a finite group, $f:Q\to Conj(G)$ is a quandle morphism and $\rho :G\to GL(V)$ a finite dimensional group representation of the group $G$ over $\mathbb{C}$, the map $\rho \circ f$ is a quandle representation that is completely reducible. Indeed, since $G$ is finite, the image of $\rho$ consists of diagonalizable automorphisms. As a concret example, the quandle morphism $Q\to Inn(Q),x\mapsto L_x$ can be composed by finite dimensional representations of $Inn(Q)$ to obtain completely reducible quandle representations of $Q$.
\end{example}

\subsection{Unitary representations}\label{S32}
\begin{definition}
A quandle representation $\rho : Q \to GL(V)$ is unitary if $V$ is endowed with an inner product for wich the elements of $\rho(Q)$ are isometries.
\end{definition}
As for groups :
\begin{proposition}
A finite dimensional unitary quandle representation is completely reducible.
\end{proposition}
\begin{proof}
Let $V$ be the vector space of the representation and $Q$ the quandle. For dimensional reasons $V$ contains a non zero irreducible subrepresentations $V_1$. If $V_1=V$, $V$ is irreducible and hence completely reducible. Otherwise, since $V_1$ is stable under the elements of $Q$ the orthogonal $V^1$ of $V_1$ is stable under the elements of $Q$. $V$ hence decomposes as the sum of two subrepresentations $V_1\oplus V^1$ with $V_1$ irreducible and $V^1$ unitary. For dimensional reasons $V^1$ has an irreducible subrepresentation $V_2$ and the orthogonal $V^2$ of $V_2$ in $V^1$ is stable under the element of $Q$. If $V^2=0$ the representation is completely reducible otherwes $V=V_1\oplus V_2 \oplus V^2$. Iterating the process if needed, we get that $V$ decomposes as a direct sum of irreducible subrepresentations.
\end{proof}
\begin{example}
Let $Q$ be a finite quandle and $R\subset Q$ such that for $x\in Q$ and $y\in R$, $x\triangleright y \in R$ $($i.e. $R$ is a union of orbits under the action of $Inn(Q)$$)$. For $x\in Q$ the left translation $L_x$ induces a permutation of $R$. Hence, it induces a unitary automorphism $\rho_R(x)$ of the free complex vector space $\mathbb{C}R$ over $R$. Since $L_{x\triangleright x'}=L_xL_{x'}L_x^{-1}$. The assignement $\rho_R :Q \to GL(\mathbb{C}R), x\mapsto\rho_R(x)$ is a unitary quandle representation. For $R=Q$ the representation $\rho_R$ and the regular representation of $Q$ introduced in \cite{rep} are dual.
\end{example}

\begin{theorem}
Let $Q$ be a finite quandle and $\rho : Q \to GL(V)$ be an irreducible finite dimensional complex representation of $Q$. $\rho$ is unitary for some inner product if and only if the determinant of every element of $\rho(Q)$ has modulus $1$.
\end{theorem}
\begin{proof}
If $\rho$ is unitary the elements of $\rho(Q)$ are isometries and hence have determinant of modulus $1$. We will prove the converse. Assume that the determinant of any element of $\rho(Q)$ has modulus $1$. Let $\rho_{G(Q)} : G(Q) \to GL(V)$ be the corresponding representation of $G(Q)$. Since the image of $\rho_{G(Q)}$ consist of products of elements of the image of $\rho$, the determinant of any element in the image of $\rho_{G(Q)}$ has modulus $1$. By proposition \ref{Zp}, we have a central extension $1 \to Z_p \to G(Q) \to G(Q)/Z_p\to 1$ with $G(Q)/Z_p$ finite. Set $H=G(Q)/Z_p$. Let $s:H\to G(Q)$ be a section of $G(Q)\to H$ and $\langle-,-\rangle$ be an inner product of $V$. We define the inner product $\langle -,-\rangle_s$ on $V$ by :
$$ \langle v,w\rangle_s= \sum_{h\in H} \langle \rho_{G(Q)}(s(h))(v), \rho_{G(Q)}(s(h))(w)\rangle,$$
for $v,w\in V$. Take $g \in G(Q)$; $g$ can be written as $zs(h_1)$ for some $z\in Z_p$ and $h_1\in H$. We have that :
$$ \langle \rho_{G(Q)}(g)(v),\rho_{G(Q)}(g)(w)\rangle_s= \sum_{h\in H} \langle \rho_{G(Q)}(z_hs(hh_1))(v), \rho_{G(Q)}(z_hs(hh_1))(w)\rangle$$
where 
$$ z_h=zs(h)s(h_1)s(hh_1)^{-1},$$
lies in $Z_p$. But $Z_p$ lies in the center and the representation $\rho_{G(Q)}$ is irreducible with all elements in the image having a determinant of modulus $1$. Hence, $\rho_{G(Q)}(z_h)=\lambda_h id$ with $\lambda_h$ of modulus $1$. This proves that  :
\begin{equation*}
 \begin{split}
\langle \rho_{G(Q)}(g)(v),\rho_{G(Q)}(g)(w)\rangle_s  & = \sum_{h\in H} \langle \lambda_h\rho_{G(Q)}(s(hh_1))(v), \lambda_h\rho_{G(Q)}(s(hh_1))(w)\rangle\\
&=\sum_{h\in H}\lambda_h\overline{\lambda_h} \langle \rho_{G(Q)}(s(hh_1))(v), \rho_{G(Q)}(s(hh_1))(w)\rangle\\
&=\sum_{h\in H} \langle \rho_{G(Q)}(s(hh_1))(v), \rho_{G(Q)}(s(hh_1))(w)\rangle\\
&= \langle v,w\rangle_s
\end{split}
\end{equation*}
Therefore, the group representation $\rho_{G(Q)}$ is unitary with respect to $\langle -,-\rangle_s$ and so is $\rho$. We have proved the theorem.
\end{proof}
\begin{definition}
Let $Q$ be a quandle. A quandle character of $Q$ is quandle morphism from $Q$ to $Conj(\C^\times)$.
\end{definition}
\begin{proposition}
Quandle characters of $Q$ correspond to maps $Q\to \C^\times$ that are constant on the orbits of $Q$.
\end{proposition}
\begin{proof}
This can be readly checked.
\end{proof}
For a complex number $z=re^{i\theta}$ with $r>0$, $ 0\leq \theta < 2\pi$ and $m$ a positive natural number, we define $z^{\frac{1}{m}}$ by :
$$z^{\frac{1}{m}}=r^{\frac{1}{m}}e^{i\frac{\theta}{m}}.$$
\begin{proposition}
Let $\rho : Q \to GL(V)$ be a quandle representation with $V$ finite dimensional over $\mathbb{C}$. The map $\chi_\rho : Q \to \mathbb{C}^\times$ given by :
$$\chi_\rho(x)=\frac{1}{det(\rho(x))^{\frac{1}{dim(V)}}},$$
for $x\in Q$, is a quandle character.
\end{proposition}
\begin{proof}
This can be readily checked
\end{proof}
\begin{proposition}
Let $\rho :Q \to GL(V)$ be an irreducible finite dimensional quandle representation over $\mathbb{C}$, with $Q$ finite. The representation $\chi_\rho \cdot \rho : Q  \to GL(V)$ defined by :
$$(\chi_\rho \cdot \rho)(x)=\chi_\rho(x)  \rho(x),$$
for $x \in Q$ and where $\chi_\rho$ is as in the previous proposition, is an irreducible unitary representation with respect to some inner product.
\end{proposition}
\begin{proof}
It is not hard to check that $\chi_\rho\cdot \rho $ is an irreducible quandle representation. The fact that it is unitary follows from the previous theorem. Indeed, the determinant of the elements of the image of $\chi_\rho \cdot \rho$ is $1$.
\end{proof}
This proves that every (finite dimensional) irreducible representation of a finite quandle is unitary up to a twist by a quandle character.

\subsection{Degree of irreducible representations}
\begin{proposition}
Let $G$ be a countable group which is a central extension of a finite group $H$. If $\rho : G\to GL(V)$ is an irreducible representation over $\C$ then :
\item[1)] $\rho$ is a finite dimensional representation.
\item[2)] The degree of $\rho$ $($dimension of $V$$)$ is bounded by $\sqrt{\vert H \vert }$.
\item[3)] The degree of $\rho$ divides $\vert H \vert$.
\end{proposition}
\begin{proof}
Let $\rho$ be an irreducible representation of $G$ as in the proposition. Any non zero vector $v$ of $V$ is cyclic and hence $V$ is the span of the vectors $\{\rho(g)(v)\}_{g\in G}$. Since $G$ is countable $V$ has countable dimension. By Dixmier’s lemma the center of $G$ acts by multiplication of scalar and hence $V$ is spanned by the vectors $\{\rho(s(h))(v)\}_{h\in H}$ where $s$ is a section of the projection $G\to H$. This proves that $V$ is finite dimensional since $H$ is finite. We have proved $1)$. Since $\rho$ is irreducible, by the density theorem (\cite{etingof}), the algebra morphism $f:\C[G]\to End(V)$ induced by $\rho$ is surjective. Since the center of $G$ acts by scalars, the image of $f$ is lineary generated by $\{\rho(s(h))\}_{h\in H}$ where $s$ is as before. This proves that $dim(End(V))=(dim(V))^2$ is at most $\vert H \vert$. This gives the bound in $2)$. Since the center of $G$ acts by scalars on $V$, $\rho$ induces and irreducible projective representation $\rho'$ of $H$. Since the projection $g: SL(V)\to PGL(V)$ is surjective and has finite central kernel the preimage $\tilde{H}$ of $\rho'(H)$ by $g$ is a finite central extension of $\rho'(H)$. The natural action of $H$ is irreducible since $\rho'$ is. By standard representation theory (\cite{serre}) $dim(V)$ divides the index  $[\tilde H: Z]$, where $Z$ is the center of $\tilde H$.  This proves that $dim(V)$ divides $[\tilde H :Z']=\vert \rho'(H)\vert$ where $Z'$ is the kernel of the projection of $\tilde H\to \rho'(H)$. Since $\vert \rho'(H)\vert$ divides $\vert H \vert$, it follows that $dim(V)$ divides $\vert H \vert$.
\end{proof}

\begin{theorem}
Let $Q$ be a finite quandle and $\rho :Q \to GL(V)$ be an irreducible quandle representation.
\item[1)] $\rho$ is a finite dimensional representation.
\item[2)] The degree of $\rho$ is bounded by $\sqrt{\vert Inn(Q) \vert }$.
\item[3)] The degree of $\rho$ divides $\vert Inn(Q) \vert$.
\end{theorem} 
\begin{proof}
Since $Q$ is finite $G(Q)$ is finitely presented and hence countable. Moreover it is a central extension of a finite group $Inn(Q)$ (proposition \ref{cent}). Hence, by the last proposition any irreducible representation of $G(Q)$ satisfy the $3$ points in the theorem. Finally, an irreducible quandle representation extend to an irreducible representation of $G(Q)$ (proposition \ref{correspondancerep}). This proves the theorem.
\end{proof}
From what we have seen in the proof of the theorem :
\begin{proposition}\label{eee}
 Irreducibles of $G(Q)$ for $Q$ finite are finite dimensional.
\end{proposition}
\begin{proposition}\label{fincor}
Let $\rho : Q\to GL(V)$ be a representation of a finite quandle $Q$. $\rho$ is irreducible if and only if $\rho_{G(Q)}$ is irreducible.
\end{proposition}
\begin{proof}
We have seen that in general $\rho$ irreducible implies that $\rho_{G(Q)}$ is irreducible (proposition \ref{correspondancerep}). Now assume $\rho_{G(Q)}$ irreducible. It is finite dimesional. A proper subspace stable by $\rho(x)$ is stable by $\rho(x)^{-1}$. Hence a stable subspace $W$ by all the elements $\rho(x)=\rho(e_x)$ for $x\in Q$ is stable by all the elements $\rho(x)=\rho(e_x)_{G(Q)}$ and $\rho(x)^{-1}=\rho(e_x)_{G(Q)}^{-1}$ for $x\in Q$. It follows that $W$ is stable by  $\rho_{G(Q)}$ since the elements $e_x$ generate $G(Q)$ and hence $W$ is either $0$ either $V$ and $\rho$ is irreducible. 
\end{proof}
\begin{example}\label{pinn}
For $p$ a prime, let $Q=\{x\}\cup \{y_i\}_{i\in \Z/p\Z}$ be the quandle where $L_{y_i}$ is the identity and $L_x$ is given by the $p$-cycle $(y_1,y_2,\dots,y_{p})$. $Inn(Q)$ is isomorphic to $\Z/p\Z$. It follows from the theorem that the only irreducible representations of $Q$ are $1$-dimensional and hence correspond to characters of $Q$. We note that theses quandles are not injective. Indeed, in the envelopin group, we have the relation $e_{y_i}e_xe_{y_i}^{-1}=e_{y_i\triangleright x}=e_x$ and hence $e_xe_{y_i}e_x^{-1}=e_{y_i}$. On the other hand $e_xe_{y_i}e_x^{-1}=e_{x\triangleright y_i}=e_{y_{i+1}}$. We deduce from the last two relations that $e_{y_i}=e_{y_j}$ for all $i,j\in \Z/p\Z$. In fact, $G(Q)$ is isomorphic to $\Z^2$ regardless if we take $p$ prime or not.
\end{example}
\subsection{faithful representations of injective quandles}
In \cite{lebVen}, the authors remark that since the enveloping group of a finite quandle is central by finite it is linear. We will exploit the central by finite property to remark few facts.

\begin{lemma}
Every finitely generated abelian group admits a faithful representation into $GL_n(\Z)$ for some $n$.
\end{lemma}
\begin{proof}
For $m>0$ denote by $\rho_m: \Z/m\Z \to GL_m(\Z)$ the representation given by :
$$ \rho_m(1)=   \begin{bmatrix}
  0& 0 	& \dots &0&1\\
1&\ddots &\ddots &\ddots&0  \\
0&\ddots & \ddots & \ddots& \vdots \\
\vdots& \ddots & \ddots & \ddots &\vdots\\
0& &0 &1 &  0\end{bmatrix}
$$
It is a faithful representation of $\Z/m\Z$. For $m=0$, denote by $\rho_0 : \Z \to GL_2(\Z)$ the representation given by :
$$\rho_0(1)= \begin{bmatrix}
  1& 1\\
0&1 \\
\end{bmatrix} $$
$\rho_0$ is a faithful representation. A finitely generated abelian group is isomorphic to :
 $$G=\Z/{m_1}\Z\times \Z/m_2\Z \times \cdots \times \Z/m_k \Z.$$
for some $m_i\geq 0$. Let $\theta_i$ be the representation of $G$ given by $$\theta_i(a_1,\dots , a_k)=\rho_{m_i}(a_i),$$ for $(a_1,\dots,a_k)\in G$. One checks that the representation : $$\overset{k}{\underset{i=1}{\bigoplus}} \theta_i,$$
is a faithful representation of $G$ into some $GL_n(\Z)$ for some $n$. We have proved the lemma.
\end{proof}
\begin{proposition}\label{faithZ}
For $Q$ be a finite quandle, there exist a faithful group representation $\rho: G(Q)\to GL_n(\Z)$ for some $n$. 
\end{proposition}
\begin{proof}
We have a central extension $\theta_Q : G(Q) \to Inn(Q)$ with kernel $K$ a finitely generated abelian group (proposition \ref{cent} and corollary \ref{corcent}) with $Inn(Q)$ finite. $K$ admit a faithful representation into $GL_m(\Z)$ for some $m$. The induced representation to $G(Q)$ of this representation of $K$ is faithful and finite dimensional since $Inn(Q)$ is finite and the representation we are inducing is faithfull. It take values in some $GL_n(\Z)$ for some finite $n$.
\end{proof}
We will say that a quandle representation is faithful if the corresponding map is injective. Recall that a quandle $Q$ is called injective if the universal map $Q\to G(Q)$ is injective.
\begin{proposition}
A finite quandle $Q$ admits a faithful finite dimensional representation over $\C$ if and only if $Q$ is injective.
\end{proposition}
\begin{proof}
If $Q$ is injective, then it admits a faithful finite dimensional representation by the last proposition. If $Q$ is not injective, a representation extend to a representation of $G(Q)$ and hence a representation of $Q$ can't be injective.
\end{proof}
\begin{lemma}
A finitely generated abelian group admits a faithful finite dimensional unitary representation over $\mathbb{C}$.
\end{lemma}
\begin{proof}
A finitely generated abelian group is a product of cyclic groups. A cyclic group admits a faithful $1$-dimensional unitary representation. One can construct a "product representation" from individual representations of the cyclic groups.
\end{proof}
\begin{proposition}
Let $Q$ be a finite quandle. The enveloping group $G(Q)$ of $Q$ admits a faithful finite dimensional unitary representation over $\mathbb{C}$.
\end{proposition}
\begin{proof}We argue as in proposition \ref{faithZ} and by using the last lemma instead of the one before.
\end{proof}
Since a unitary finite dimensional representation is completely reducible, we have :
\begin{corollary}
A finite injective quandle admits a faithful finite dimensional completely reducible representation over $\C$.
\end{corollary}
It follows that :
\begin{corollary}
The irreducible representations over $\C$ of a finite injective quandle $Q$ seperate the quandle; meaning that if $x$ and $y$ are distinct elements of $Q$ there is an irreducible representation $ \rho$ of $Q$ such that $\rho(x)\neq \rho(y)$.
 \end{corollary}
A quandle is called trivial if $x\triangleright y=y$ for all $x,y$ in the quandle.
\begin{corollary}
A non trivial injective finite quandle admits an irreducible representation of degree strictly greater than $1$.
\end{corollary}
\begin{proof}
$1$-dimensional representation correspond to quandle characters. Since characters are constant on orbits, they seperate points if and only if the orbits are reduced to single points. This means that the quandle is trivial. Hence, the only injective quandle with only $1$-dimensional irreducible representations are trivial quandles.
\end{proof}
\begin{proposition}
Let $Q$ be a finite quandle, the irreducible quandle representations of $Q$ over $\mathbb{C}$ are $1$-dimensional if and only if $G(Q)$ is abelian.
\end{proposition}
\begin{proof}
Since we have a correspondance between irreducibles of $Q$ and $G(Q)$ for finite quandles (proposition \ref{fincor}), the irreducible representations of $Q$ are $1$-dimensional if and only if those of $G(Q)$ are $1$-dimensional. Now assume that $G(Q)$ is abelian. An irreducible finite dimensional of $G(Q)$ for $Q$ finite is finite dimensional (see proposition \ref{eee}) and hence it follows from Schur's lemma that such representation is $1$-dimensional. Now assume that all irreducibles of $G(Q)$ are $1$-dimensional. By the last proposition $G(Q)$ admits a faithful unitary finite dimensional representation. It hence decomposes into $1$-dimensional representations. This gives an injective morphism from $G(Q)$ into the group $(\C^\times)^n$ for some $n$. This proves that $G(Q)$ is abelian.  
\end{proof}
In \cite{lebMor}, the authors describe finite quandles with abelian enveloping group.

\subsection{Irreducible representations of some quandles}
Let $n$ and $m$ be two nonzero natural numbers. Let $Q_{n,m}$ be the set $$\{x_i \vert i \in \mathbb{Z}/n\mathbb{Z} \} \cup \{y_i \vert i \in \mathbb{Z}/m\mathbb{Z} \}$$ equipped with the binary operation defined by :
$$x_i \triangleright y_j=y_{j+1},\quad y_i\triangleright x_j=x_{j+1} ,\quad x_{i}\triangleright x_j=x_j,\quad y_i\triangleright y_j=y_j.$$
With the above binary operation $Q_{n,m}$ is a quandle. These quandles appear, for instance, in \cite{EM}. 

\begin{proposition}\label{ir}
Let $\rho : Q_{n,m} \to GL(V)$ be a quandle representation. $\rho$ is irreducible if and only if $\rho(x_1)$ and $\rho(y_1)$ has no common stable subsbace other than $0$ and $V$.
\end{proposition}
\begin{proof}
We have seen that irreducible of finite quandles are finite dimension. Hence, $V$ is finite dimensional. If $\rho(x_1)$ and $\rho(y_1)$ has no common stable subsbace other then $0$ and $V$ then the representation is clearly irreducible. Now assume that $\rho(x_1)$ and $\rho(y_1)$ has a common stable subsbace $W$ other than $0$ and $V$. We have that $\rho(x_i)=\rho(y_1)^{i-1}\rho(x_1)\rho(y_1)^{1-i}$ and $\rho(y_i)=\rho(x_1)^{i-1}\rho(y_1)\rho(x_1)^{1-i}$. Hence, $W$ is stable under any element of $\rho(Q_{n,m})$. This proves that the representation is not irreducible. Therefore, if the representation is irreducible then $\rho(x_1)$ and $\rho(y_1)$ has no common stable subsbace other than $0$ and $V$. We have proved the equivalence.
\end{proof}
Recall that we have seen that quandle irreducible representations are finite dimensional.
\begin{proposition}\label{ir2}
Let $\rho :Q_{n,m} \to GL(V)$ be an irreducible quandle representation over $\mathbb{C}$. If $V$ is not $1$-dimensional then :
\item[1)]  For some primitive $d$-th root of unity $\alpha$ with some $d$ dividing both $m$ and $n$, ($d>1$) :
$$\rho(y_1)\rho(x_1)\rho(y_1)^{-1}=\alpha \rho(x_1)$$
\item[2)] $\rho(y_1)^d=\lambda id$ for some $\lambda \in \mathbb{C}^\times$.
\item[3)] If $v$ is an eigenvector of $\rho(x_1)$ associated to the eigenvalue $\beta$, then the family 
$$\{v,\rho(y_1)(v),\dots,\rho(y_1)^{d-1}(v)\}$$ form a basis of $V$ and $$\rho(x_1)(\rho(y_1)^i(v))=\frac{\beta}{\alpha^i}\rho(y_1)^i(v),$$
for $i=0,\dots,d-1$. 
\end{proposition}
\begin{proof}
Recall that we have a group morphism $\theta_{Q_{n,m}}: G(Q_{n,m})\to Inn(Q_{n,m})$ mapping $e_x$ to $L_x$ and with central kernel. We have that $$\theta_{Q_{n,m}}(e_{x_1}e_{ y_1\triangleright x_1}^{-1})=L_{x_1}L_{y_1\triangleright x_1}^{-1}=L_{x_1}L_{x_2}^{-1}=L_{x_1}L_{x_1}^{-1}=1.$$
Hence, $e_{x_1}e_{ y_1\triangleright x_1}^{-1}$ lies in the center of $G(Q_{n,m})$. Since the representation is irreducible $e_{x_1}e_{ y_1\triangleright x_1}^{-1}$ acts by multiplication by scalar in $V$ and hence $$\rho_{G(Q)}(e_{x_1}e_{ y_1\triangleright x_1}^{-1})=c\mathrm{id},$$ for some $c$ in $\C^\times$. Now, $$\rho_{G(Q)}(e_{x_1}e_{ y_1\triangleright x_1}^{-1})=\rho_{G(Q)}(e_{x_1})\rho_{G(Q)}(e_{y_1}e_{x_1}^{-1}e_{y_1}^{-1})=\rho(x_1)\rho(y_1)\rho(x_1)^{-1}\rho(y_1)^{-1}.$$
Combining the last two equations, we get that :
 $$\rho(y_1)\rho(x_1)\rho(y_1)^{-1}=\alpha \rho(x_1),$$
for some $\alpha \in \mathbb{C}^\times$. We will write $y_1^n\triangleright x_1$ for the iteration $n$ times of $y_1\triangleright$ on $x_1$ : $$y_1\triangleright (y_1\triangleright( \cdots (y_1\triangleright x_1)\cdots)).$$
We have that : 
$$ \rho(y_1^n\triangleright x_1)=\rho(x_1),$$
and by the equation $\rho(y_1)\rho(x_1)\rho(y_1)^{-1}=\alpha \rho(x_1)$ we obtained before, we have that :
$$\rho(y_1^n\triangleright x_1)=\rho(y_1)^n\rho(x_1)\rho(y_1)^{-n}=\alpha^n\rho(x_1).$$
Therefore $\alpha^n=1$. This proves that $\alpha$ is an $n$-root of unity. Manipulating the equation $\rho(y_1)\rho(x_1)\rho(y_1)^{-1}=\alpha \rho(x_1)$, we get that :
$$\rho(x_1)\rho(y_1)\rho(x_1)^{-1}=\alpha^{-1} \rho(y_1).$$
Applying the same reasoning as just before with $\rho(x_1^m\triangleright y_1)$ we get that $\alpha$ is an $m$-th root of unity. This proves that $\alpha$ is a $n$-th and $m$-th root of unity. We still need to prove that $\alpha \neq 1$ to complete the proof of $1$. Assume $\alpha=1$. This implies that $\rho(x_1)$ and $\rho(y_1)$ commute and hence any eigenspace of these two automorphisms of $V$ is a common stable subspace. It follows from the proposition \ref{ir} since $\rho $ is irreducible that $\rho(x_1)$ and $\rho(y_1)$ are scalar multiple of the identity. But this contradict the same proposition since $V$ is not $1$-dimensional. Hence, $\alpha \neq 1$. This completes the proof of $1)$. We now prove $2)$. The equation in $1)$ of the proposition proves that $\rho(y_1)^d$ commutes with $\rho(x_1)$. It obviously also commute with $\rho(y_1)$. Hence, an eigenspace of $\rho(y_1)^d$ is a stable subspace for $\rho(y_1)$ and $\rho(x_1)$. It follows from proposition \ref{ir} since $\rho$ is irreducible that an eigenspace of $\rho(y_1)^d$ is equal to $V$. This proves that $\rho(y_1)^d$ is a scalar multiple of the identity. We prove $3)$. Using the equation in $1)$ one shows that if $w$ is associated to the eigenvalue $\theta$ with respect to $\rho(x_1)$ then $\rho(y_1)w$ is an eigenvector for $\rho(x_1)$ associated to the eigenvalue $\frac{\theta}{\alpha}$. Applying, this result and an induction one gets the equation : 
$$\rho(x_1)(\rho(y_1)^i(v))=\frac{\beta}{\alpha^i}\rho(y_1)^i(v),$$
of $3)$. Since $\alpha$ is a primitive root it follows that the vectors of the family $$\{v,\rho(y_1)(v),\dots,\rho(y_1)^{d-1}(v)\}$$ are associated to different eigenvalues of $\rho(x_1)$ and hence the family is free. Denote by $W$ the subspace spanned by this family. $W$ is stable under $\rho(x_1)$ since it consist of eigenvectors of $\rho(x_1)$. It follows from $2)$ of this proposition that $W$ is stable by $\rho(y_1)$. Hence, $W$ is stable by $\rho(x_1)$ and $\rho(y_1)$ and therefore by the proposition \ref{ir} $W=V$ since $\rho$ is irreducible. This proves that the family in $3)$ of this proposition generates $V$. We have already seen that it is free. This completes the proof of $3)$.
\end{proof}
\begin{proposition}
Let $d$ be a positive integer with $d>1$ dividing $n$ and $m$, $\alpha$ be a primitive $d$-th root of the unity and $\lambda, \beta \in \mathbb{C}^\times$ :
\item[1)] There exists a unique quandle representation $\rho_{\alpha,\lambda,\beta} : Q_{n,m} \to GL_d(\mathbb{C})$, such that :
 $$\rho_{\alpha,\lambda,\beta}(x_1)=\begin{pmatrix}
    \beta & 0   & \dots& \dots &  0 \\
     0 &\frac{\beta}{\alpha}   & \ddots&  &\vdots \\
    \vdots & \ddots &  \ddots & \ddots &\vdots \\
\vdots &   & \ddots  &\ddots &0\\
   0 &   \dots & \dots  & 0&\frac{\beta}{\alpha^{d-1}}
\end{pmatrix} \quad   \rho_{\alpha,\lambda,\beta}(y_1)=\begin{pmatrix}
   0 & 0   & \dots& 0 &  \lambda \\
     1 & \ddots   & \ddots&  &0\\
    0 & \ddots &  \ddots & \ddots &\vdots \\
\vdots &  \ddots & \ddots  &\ddots &0\\
   0 &   \dots & 0 & 1&0
\end{pmatrix}.$$
\item[2)] The quandle representation $\rho_{\alpha,\lambda,\beta}$ is given by the following : 
$$\rho _{\alpha,\lambda,\beta}(x_i)=\alpha^{i-1}\rho_{\alpha,\lambda,\beta}(x_1) \quad \text{and} \quad \rho_{\alpha,\lambda,\beta}(y_j)=\alpha^{1-j}\rho_{\alpha,\lambda,\beta}(y_1),$$
for $i\in \mathbb{Z}/n\mathbb{Z}$ and $j\in \mathbb{Z}/m\mathbb{Z}$.
\item[3)] The commutator of $(\rho(y_1),\rho(x_1))$ is equal to $\alpha id$.
\item[4)] The representation $\rho_{\alpha,\lambda,\beta}$ is irreducible.
\end{proposition}
\begin{proof}
Let $A$ be the matrix corresponding to $\rho_{\alpha,\lambda,\beta}(x_1)$ and $B$ the one corresponding to $\rho_{\alpha,\lambda,\beta}(y_1)$ in $1)$. One checks that $BA=\alpha AB$ this implies that $AB=\frac{1}{\alpha}BA$. These equations with the defining binary operation of $Q_{n,m}$ allow to show that $\rho_{\alpha,\lambda,\beta}$ defined by equations of $1)$ and $2)$ is indeed a quandle representation. This proves the existence of $\rho_{\alpha,\lambda,\beta}$ and that $2)$ follows from the uniqueness in $1)$. The uniqueness in $1)$ follows from the fact that 
$$\rho_{\alpha,\lambda,\beta}(x_i)=\rho_{\alpha,\lambda,\beta}(y_1^{i-1}\triangleright x_1)=\rho_{\alpha,\lambda,\beta}(y_1)^{i-1}\rho_{\alpha,\lambda,\beta}(x_1)\rho_{\alpha,\lambda,\beta}(y_1)^{1-i},$$
and $$\rho_{\alpha,\lambda,\beta}(y_j)=\rho_{\alpha,\lambda,\beta}(x_1^{j-1}\triangleright y_1)=\rho_{\alpha,\lambda,\beta}(x_1)^{j-1}\rho_{\alpha,\lambda,\beta}(y_1)\rho_{\alpha,\lambda,\beta}(x_1)^{1-j},$$
where we denote the $l$-times iteration of $x\triangleright$ on $y$ by $x^l\triangleright y$. We have proved $1)$ and $2)$. $3)$ follows from the equation $BA=\alpha AB$ we obtained earlier. Now let $W$ be a subspace stable under $\rho(x_1)$ and $\rho(y_1)$. If $W\neq 0$, then $W$ contains an eigenvector $v$ of $\rho(x_1)$ (the restriction of $\rho(x_1)$ to $W$ has an eigenvalue). But as one checks the family $\{v,\rho_{\alpha,\lambda,\beta}(y_1)(v),\dots,\rho_{\alpha,\lambda,\beta}(y_1)^{d-1}(v) \}$ spans $V$. Hence, $W=V$ if $W\neq 0$. It follows from proposition \ref{ir}, that $\rho_{\alpha,\lambda,\beta}$ is irreducible. We have proved $4)$ and completed the proof of the proposition.
\end{proof}
\begin{theorem}
\item[1)] The irreducible quandle representations of $Q_{n,m}$ of dimension greater then $1$ are up to equivalence the representations $\rho_{\alpha,\lambda,\beta}$ as in the previous proposition. 
\item[2)] Two representations $\rho_{\alpha,\lambda,\beta}$ and $\rho_{\alpha',\lambda',\beta'}$ are equivalent if and only if $\alpha=\alpha'$, $\lambda = \lambda'$ and  $\beta=\beta'\alpha^i$ for some $i$.
\end{theorem}
 \begin{proof}
$1)$ follows from $1)$ and $4)$ of the previous proposition and $2)$ and $3)$ of proposition \ref{ir2}. Assume $\rho_{\alpha,\lambda,\beta}$ and $\rho_{\alpha',\lambda',\beta'}$ are equivalent. It follows from $3)$ of the previous proposition that $\alpha=\alpha'$. $\lambda$ and $\lambda'$ are determined by the determinants of $\rho_{\alpha,\lambda,\beta}(y_1)$ and $\rho_{\alpha',\lambda',\beta'}(y_1)$. Hence, $\lambda=\lambda'$. Now $\beta$ and $\beta'$ are elements of the spectrum of $\rho_{\alpha,\lambda,\beta}(x_1)$ and $\rho_{\alpha',\lambda',\beta'}(x_1)$. Hence, $\beta=\beta'\alpha^i$ for some $i$. To complete the proof of $2)$ we still need to show that $\rho_{\alpha,\lambda,\beta}$ and $\rho_{\alpha,\lambda,\beta\alpha^i}$ are equivalent. Let $v$ be the eigenvector of $\rho_{\alpha,\lambda,\beta}(x_1)$ associated to $\beta \alpha^i$. The representation $\rho_{\alpha,\lambda,\beta}$ taken in the basis $\{v,\rho_{\alpha,\lambda,\beta}(y_1)(v),\dots,\rho_{\alpha,\lambda,\beta}(y_1)^{d-1}(v)\}$ is exactly $\rho_{\alpha,\lambda,\beta\alpha^i}$ ($\rho_{\alpha,\lambda,\beta}(y_1)^d=\lambda$). Indeed,  by $1)$ of the previous proposition, one only need to check this for the images of $x_1$ and $y_1$ and this is straightforward. This completes the proof of the theorem.

\end{proof}

\section{Quandles representations, Inner automorphism groups and representation groups}
This section is divided into $3$ subsections. In subsection $1$. We relate irreducible representations of a finite quandle $Q$ to irreducible projective representations of the inner automorphism group $Inn(Q)$ of the quandle. We remark that if the Schur multiplier of $Inn(Q)$ is trivial, then the irreducibles of $Q$ can be constructed using quandle characters and irreducible linear representations of $Inn(Q)$. Examples for this situation are given. We also provide a result for irreducibles of the conjugacy quandles of Schur covers. In the second subsection, we introduce the notion of representation group of a finite quandle wich is a type of finite quotient of the enveloping group. We show that the irreducible representations of a finite quandle can be constructed from irreducible linear representations of a representation group and quandle characters. In the last subsection, we give an algorithm to determine up to equivalence all the irreducible representations of a finite quandle. 
 
\subsection{Irreducible quandle representations}
Recall (see subsection \ref{S32}) that a quandle character is a quandle morphism from the quandle to $Conj(\C^\times)$ and that character correspond to maps into $\C^\times$ that are constant on the orbits. 

\begin{definition}
For  a quandle representaion $\rho$ of $Q$ and a quandle character $\chi$ of $Q$, we define the product of $\chi$ and $\rho$ by :
$$(\chi \cdot \rho)(x)=\chi(x)\rho(x),$$ 
for $x\in Q$. 
\end{definition}
\begin{proposition}
For $\chi$ and $\rho$ as in the previous definition $(\chi\cdot \rho)$ is a quandle representation, it is irreducible if and only if $\rho$ is irreducible.
\end{proposition}
\begin{proof}
We leave the proof for the reader.
\end{proof}

\begin{proposition}\label{rhoInn}
Let $Q$ be a finite quandle and let $\rho : Q \to GL(V)$ be an irreducible representation.
\item[1)] There is a unique irreducible projective representation $\rho_{Inn(Q)} : Inn(Q) \to PGL(V)$ such that : $$\rho_{Inn(Q)}(L_x)=\overline{\rho(x)},$$ for $x\in Q$ and where $\overline{\rho(x)}$ is the class of $\rho(x)$ in $PGL(V)$.
\item[2)] Let $\rho':Q \to GL(W)$ be a quandle representation equivalent to $\rho$. $\rho_{Inn(Q)}$ and $\rho'_{Inn(Q)}$ are projectively equivalent.
\item[3)] Let $\rho': Q \to GL(V)$ be an irreducible representation such that $\rho'_{Inn(Q)}=\rho_{Inn(Q)}$. We have that  $\rho'=\chi \cdot \rho$ for some character $\chi$ of $Q$.
\item[4)] Assume that $\rho_{Inn(Q)}$ lifts to a linear representation $\rho':Inn(Q) \to GL(V)$, then $\rho=\chi \cdot (\rho'\circ \theta)$ for some character $\chi$ of $Q$ and where $\theta$ is the quandle morphism $Q\to Inn(Q), x\mapsto L_x$.
\end{proposition}
\begin{proof}
Recall that we have a central extension $\theta_Q: G(Q)\to Inn(Q)$ defined by the morphism given by $e_x\to L_x$.  Denote by $Z_0$ the kernel of $\theta_Q$. We have seen that an irreducible representation of a finite quandle is finite dimensional. Let $\rho_{G(Q)}$ be the representation of $G(Q)$ associated to $\rho$. Since $\rho$ is irreducible $\rho_{G(Q)}$ is irreducible. By Schur's lemma, $Z_0$ acts by multiplication by scalar since $V$ is finite dimensional and $Z_0$ is central. Hence, the morphism $\bar{\rho}_{G(Q)}:G(Q) \to PGL(V)$ obtained by composing $\rho_{G(Q)}$ by the projection $GL(V) \to PGL(V)$ factors throught a morphism $\rho_{Inn(Q)} : Inn(Q) \to PGL(V)$. We have :
$$\rho_{Inn(Q)}(L_x)=\bar{\rho}_{G(Q)}(e_x)=\overline{\rho_{G(Q)}(e_x)}=\overline{\rho(x)},$$
for $x \in Q$ and where the overline is for the classes in $PGL(V)$. The proves the existence and uniqueness of $\rho_{Inn(Q)}$ in $1)$ since left translations generate $Inn(Q)$. Finally, since $\rho$ is irreducible $\rho_{G(Q)}$ is irreducible and this implies that $\rho_{Inn(Q)}$ is irreducible. We have proved $1)$. We prove $2)$. If $h$ is an intertwining between $\rho$ and $\rho'$ then its projectivisation is an intertwining operator between $\rho_{Inn(Q)}$ and $\rho'_{Inn(Q)}$ by the formula in $1)$ and the fact that translations generate $Inn(Q)$. This proves $2)$. We prove $3)$. If $\rho'_{Inn(Q)}=\rho_{Inn(Q)}$ then it follows from $1)$ that for $x,y \in Q$, we have :
$$ \rho'(x)=\chi(x)\rho(x), \quad \rho'(y)=\chi(y) \rho(y) \quad\text{ and } \rho'(x\triangleright y)=\chi(x\triangleright y) \rho(x\triangleright y),$$
for some $\chi(x),\chi(y),\chi(x\triangleright y)\in \mathbb{C}^\times$. Now :
$$\rho'(x\triangleright y)=\rho'(x)\rho'(y)\rho'(x)^{-1}=\chi(y)\rho(x)\rho(y)\rho(x)^{-1}=\chi(y) \rho(x\triangleright y).$$
This proves that $\chi(x\triangleright y)\rho(x\triangleright y)=\chi(y)\rho(x\triangleright y)$ and $3)$ follows. We prove $4)$. The representation $\rho'$ is irreducible since it lifts an irreducible projective representation. The quandle representation $\rho'\circ \theta$ is hence irreducible since the image of $\theta$ generates $Inn(Q)$. By $1)$ of this proposition, we have for $x$ in $Q$ :
$$(\rho'\circ\theta)_{Inn(Q)}(L_x)=\overline{(\rho'\circ\theta)(x)}=\overline{\rho'(L_x)}=\rho_{Inn(Q)}(L_x),$$
where the overline is for classes in $PGL(V)$. This proves that $(\rho'\circ\theta)_{Inn(Q)}=\rho_{Inn(Q)}$ since left translation generate $Inn(Q)$. We can therefore apply $3)$ of this proposition to obtain $4)$.

\end{proof}

\begin{proposition}
Let $Q$ be a finite quandle. Assume that $Inn(Q)$ has trivial Schur multiplier. The irreducible representations of $Q$ are the representations $\chi \cdot (\rho'\circ \theta),$ where $\rho'$ is an irreducible linear represenation of the group $Inn(Q)$ and $\chi$ is a character of $Q$.
\end{proposition}
\begin{proof}
Let $\rho'$ be an irreducible representation of $Inn(Q)$. The quandle representation $\rho'\circ \theta$ is irreducible since the image of $\theta$ generates the group $Inn(Q)$. It follows that $\chi \cdot (\rho'\circ \theta) $ is irreducible. Now let $\rho:Q\to GL(V)$ be an irreducible representation of $Q$. Since we assume that $Inn(Q)$ has trivial Schur multiplier then $\rho_{Inn(Q)}$ lifts to an irreducible linear representation $\rho':Inn(Q)\to GL(V)$. The proposition hence follows from $4)$ of the previous proposition
\end{proof}
\begin{example}
Let $Q$ be the subquandle of $Conj(S_3)$ $($symmetric group$)$ consisting of transposition. It is a quandle on $3$ elements. The inner automorphism group of $Q$ is $S_3$ and $($\cite{Sc}$)$ $S_3$ has trivial Schur multiplier. The linear irreducible representations of $Q$ can be hence constructed out of irreducibles of $S_3$ and quandle characters that are constant map into $\C^\times$.
\end{example}
\begin{example}
If $Q$ is a quandle with $Inn(Q)$ cyclic then the irreducibles of $Q$ are $1$-dimensional. Indeed, the Schur multiplier of a cyclic group is trivial.
\end{example}
\begin{example}
Let $G$ be one of the groups $D_{2n}$ $($dihedral group with $2n$ elements$)$,  $D_{4n}$ or $Q_{4n}$ $($the generalised quaternion group with $4n$ elements$)$ with $n$ odd. We have $Inn(Conj(G))=Inn(G)\simeq D_{2n}$. The group $D_{2n}$ for $n$ odd has trivial Schur multiplier $($\cite{Sc}$)$ and hence the irreducible quandle representations of $Conj(G)$ can be constructed out of quandle characters and irreducible linear representations of the group $D_{2n}$.  
\end{example}
\begin{example}
Let $M$ an abelian group and $T$ be an automorphism of $M$. The Alexander quandle $Al(M,T)$ is the set $M$ equipped with the binary operation $x\triangleright y=(1-T)x+Ty .$ The inner automorphisms group of $Q=Al(M,T)$ is the semidirect product $((1-T)M)\rtimes \langle T\rangle$, where $\langle  T\rangle$ is the subgroup of automorphisms generated by $T$ and where $T$ acts by evaluation. If $M$ is finite cyclic and $\vert M \vert$ is coprime to the order of $T$, then the Sylow subgroups of $Inn(Q)$ are all cyclic and hence $($\cite{Sc}$)$ the Schur multiplier of $Inn(Q)$ is trivial. Under the last assumptions the quandle representations of $Al(M,T)$ can be determined from irreducibles of the group $((1-T)M)\rtimes \langle T\rangle$ and quandle characters. 
\end{example}

\begin{proposition}\label{repSchur}
Let $G$ be a Schur cover. The irreducible quandle representations of $Conj(G)$ are the representations $\chi\cdot \rho$, where $\rho$ is an irreducible linear representation of the group $G$ and $\chi$ is a character of $Conj(G)$.
\end{proposition}
\begin{proof}
It is not hard to prove that the representations in the theorem are irreducible. We need to prove that every irreducible representation is of that form. To simplify the notations we will use $Inn(Q)$ instead of $Inn(Conj(G))$. Assume $G$ is a Schur cover of $H$. In particular, $G$ is a central extension of $H$. Since $Inn(Q)=Inn(G)\simeq G/Z(G)$, where $Z(G)$ is the center of $G$, we have a commutative diagramm : 
$$\begin{tikzcd}
G \arrow{d}{\alpha}\arrow{r}{\theta} & Inn(Q)  \\ 
H  \arrow{ru}{\beta} & 
\end{tikzcd} .$$ 
Let $\rho$ be an irreducible representation of $Conj(G)$ over a vector space $V$. $\rho_{Inn(Q)}\circ \beta$ is an irreducible projective representation of $H$. Since $G$ is a Schur cover of $H$, there is a linear irreducible group representation $\rho':G\to GL(V)$ such that for $g\in G$:
$$\overline{\rho'(g)}=(\rho_{Inn(Q)}\circ\beta)\circ\alpha = \rho_{Inn(Q)}\circ \theta (g).$$
Combining the last equation with $1)$ of proposition \ref{rhoInn}, we get for $g\in G$ :
$$\rho'_{Inn(Q)}(L_g)=\overline{\rho'(g)}=\rho_{Inn(Q)}\circ\theta(g)=\rho_{Inn(Q)}(L_g).$$
Hence, $\rho'_{Inn(Q)}=\rho_{Inn(Q)}$. It follows by $3)$ of proposition \ref{rhoInn} that $\rho=\chi \cdot \rho'$ for some character $\chi $ of $Conj(G)$. This proves the proposition.
\end{proof}
\begin{remark}\label{Schur}
As we have seen in the proof any projective representation of $Inn(G)$ lift to a linear representation of $G$ a Schur cover. This proves that any irreducible projective representation of $Inn(Conj(G))=Inn(G)$ is induced by an irreducible representation of $Conj(G)$ if $G$ is a Schur cover.
\end{remark}
\begin{example}
The groups $Q_{4n}$ and $D_{4n}$ for $n$ even are Schur covers of $D_{2n}$. Hence, the irreducible quandle representation of $Conj(Q_{4n})$ $($resp. $Conj(D_{4n})$ $)$ for $n$ even can be constructed out of the irreducible representations of $Q_{4n}$ $($resp. $D_{4n}$$)$.
\end{example}
\begin{remark}
We have $Inn(Q_{4n}) \simeq D_{2n}$ and the Schur multiplier of $D_{2n}$ for $n$ even is $\Z/2\Z$. Hence, $D_{2n}$ has a non linearizable projective representation for $n$ even. This $($$Conj(Q_{4n})$, $n$ even$)$ provides an example of a quandle irreducible representation lifting a non linearizable projective representation of its inner automorphism group. This irreducible quandle representation is not equivalent to a quandle representation lifting a linearizable projective representation of the inner automorphism group.
\end{remark}

\subsection{Representation groups}
Recall that for $Q$ a quandle, we have a central extension $\theta_{Q} : G(Q) \to Inn(Q)$ given by $e_x\mapsto L_x$ for $x\in Q$. We will denote by $Z_0$ the kernel of $\theta_Q$ and by $Tor(Z_0)$ the torsion subgroup of $Z_0$. Recall that if $Q$ is finite $Z_0$ is a finitely generated abelian group.

\begin{definition}
Let $Q$ be a finite quandle and  $\alpha : \Z^n\bigoplus Tor(Z_0) \to Z_0$ be an isomorphism. The representation group $G(Q)_\alpha$ of $Q$ associated to $\alpha$ is the quotient $G(Q)/\alpha(\Z^n)$.
\end{definition}
We will denote by $q_\alpha$ the quotient map $G(Q)\to G(Q)_\alpha$ and by $\bar{\theta}_Q$ the morphism $G(Q)_\alpha \to Inn(Q)$ factoring $\theta_Q$.
 \begin{proposition}
Let $G(Q)_\alpha$ be a representation group.
\item[1)] The morphism $\bar{\theta}_{Q}: G(Q)_\alpha \to Inn(Q)$ is a central stem extension with nucleus isomorphic to the finite group $Tor(Z_0)$.
\item[2)] The group $G(Q)_\alpha$ is finite.
\end{proposition}
\begin{proof}
We prove $1)$. The fact that $\bar{\theta}_Q$ defines a central extension with nucleus isomorphic to the finite group ($Z_0$ is finitely generated) $Tor(Z_0)$ is obvious. The kernel of $\bar{\theta}_Q$ is $q_\alpha(Tor(Z_0))$. $Tor(Z_0)$ is finite and hence is a subgroup of the derived subgroup $G(Q)'$ of $G(Q)$ ($2)$ of proposition \ref{abder}). Hence, the nucleus of $\bar{\theta}_Q$ equal to $q_\alpha(Tor(Z_0))$ is a subgroup of the derived group of $G(Q)_\alpha$. This proves the central extension is stem. Since $G(Q)_\alpha$ is the extension of the finite group $Inn(Q)$ by a finite nucleus, $G(Q)_\alpha$ is finite.
\end{proof}
 For a multiplicative character $\chi$ of a group $G$ and a representation $\rho :G \to GL(V)$, we define the linear representation $\chi\cdot \rho : G \to GL(V)$ by:
$$(\chi \cdot \rho) (g)= \chi(g)\rho(g),$$
for $g\in G$. 

\begin{proposition}\label{tc}
Let $Q$ be finite and $\rho : G(Q) \to GL(V)$ be an irreducible representation of the group $G(Q)$ and $\alpha :\Z^n\bigoplus Tor(Z_0)\to Z_0$ be an isomorphism. 
\item[1)] There exists a multiplicative character $\chi$ of $G(Q)$ such that the kernel of $\chi \cdot \rho$ contains $\alpha(\Z^n)$. 
\item[2)] The representation $\chi \cdot \rho$ induces an irreducible representation $\rho_\chi : G(Q)_\alpha \to GL(V)$ and $\rho=\chi^{-1}\cdot (\rho_\chi\circ q_\alpha)$.
\end{proposition}
\begin{proof}
We have seen that irreducibles of $G(Q)$ for $Q$ finite are finite dimensional. $\rho$ is therefore finite dimensional. $\alpha(\Z^n)$ lies in the center of $G(Q)$ and hence acts by scalars in the representation. Let $\lambda : \alpha(\Z^n)\to \mathbb{C}^\times$ be the associated multiplicative character. Since the derived group of $G(Q)$ is finite the $\alpha(\Z^n)$ injects into $G(Q)^{ab}$ and hence the character $\lambda$ extends to a multiplicative character $\tilde{\lambda}:G(Q)\to \mathbb{C}^\times$. We obtain the proposition by taking $\chi=\tilde{\lambda}^{-1}$.
\end{proof}
We will denote by $\varphi_{Q,\alpha}$ the composition of $\varphi_Q:Q\to G(Q)$ with the quotient map $q_\alpha:G(Q)\to G(Q)_\alpha$.
\begin{theorem}\label{class}
Let $Q$ be a finite quandle and $G(Q)_\alpha$ be a representation group of $Q$. The irreducible quandle representations of $Q$ are the representations of the form $\chi \cdot( \rho \circ \varphi_{Q,\alpha})$, where $\chi$ is a character of $Q$ and $\rho$ is an irreducible representation of the finite group $G(Q)_\alpha$.
\end{theorem}
\begin{proof}
The fact that representations of this form are irreducible follows from the fact that the image of $\varphi_{Q,\alpha}$ generates $G(Q)_\alpha$. Now, let $\rho'$ be an irreducible representation of $Q$. $\rho'$ extends to an irreducible group representation $\rho_{G(Q)}$ of $G(Q)$ such that :
 $$\rho'=\rho'_{G(Q)}\circ \varphi_Q.$$ 
By the previous proposition :
$$\rho'_{G(Q)}=\lambda^{-1}\cdot (\rho''\circ q_\alpha),$$ 
for some mutiplicative character $\lambda$ of $G(Q)$ and some irreducible representation $\rho''$ of $G(Q)_\alpha$. Combining the last two equations we get :
$$\rho'=(\lambda^{-1}\circ \varphi_Q)\cdot(\rho''\circ \varphi_{Q,\alpha}).$$
This proves that $\rho'$ is of the form as in the theorem since $\lambda^{-1}\circ \varphi_Q$ is a character of $Q$. We have proved the theroem.
 \end{proof}
 
\begin{proposition}\label{eeeeee}
Let $\chi \cdot (\rho\circ \varphi_{Q,\alpha})$ and $\chi'\cdot (\rho' \circ \varphi_{Q,\alpha})$ be as in the previous theorem. Let $\chi_{\rho}$ and $\chi_{\rho'}$ be the respctive characters of the representations $\rho$ and $\rho'$.  $\chi \cdot (\rho\circ \varphi_{Q,\alpha})$ and $\chi'\cdot (\rho' \circ \varphi_{Q,\alpha})$ are equivalent if and only if $\chi'=(\chi_{gr}\circ \varphi_{Q,\alpha})\chi$ for some multiplicative character $\chi_{gr}$ such that $\chi_{\rho}=\chi_{gr}\chi_{\rho'}$.
\end{proposition}
\begin{proof}
Assume that $\chi_{\rho}=\chi_{gr}\chi_{\rho'}$ and that $\chi'=(\chi_{gr}\circ \varphi_{Q,\alpha})\chi$. The assumption $\chi_{\rho}=\chi_{gr}\chi_{\rho'}$ implies that $\rho$ is equivalent to $\chi_{gr}\cdot\rho'$. The assumption $\chi'=(\chi_{gr}\circ \varphi_{Q,\alpha})\chi$, give that :
$$\chi' \cdot (\rho'\circ \varphi_{Q,\alpha})=(\chi_{gr}\circ \varphi_{Q,\alpha})\chi(\rho'\circ \varphi_{Q,\alpha})=\chi\cdot((\chi_{gr}\cdot\rho')\circ \varphi_{Q,\alpha}).$$
It follows from the last equation and the equivalence between $\rho$ and $\chi_{gr}\cdot\rho'$, that  $\chi \cdot (\rho\circ \varphi_{Q,\alpha})$ and $\chi'\cdot (\rho' \circ \varphi_{Q,\alpha})$ are equivalent. 
Now assume that $\chi \cdot (\rho\circ \varphi_{Q,\alpha})$ and $\chi'\cdot (\rho' \circ \varphi_{Q,\alpha})$ are equivalent. We have that, there exists an $f$ such that :
$$f\rho(g)f^{-1}=\lambda_g\rho'(g)$$
for all $g$ in the image of $\varphi_{Q,\alpha}$ with $\lambda_g$ an invertible constant. Since the image of $\varphi_{Q,\alpha}$ generates $G(Q)_\alpha$. We get using multipicativity of $\rho$ and $\rho'$ that :
$$f\rho(g)f^{-1}=\lambda_g\rho'(g)$$
for all $g$ in $G(Q)_\alpha$ and that $\chi_{gr} :g\to \lambda_g$ is a multiplicative character. Hence by taking the trace we get that $\chi_{\rho}=\chi_{gr}\chi_{\rho'}$. For $x$ in $Q$, $\lambda_{\varphi_{Q,\alpha}(x)}=\chi'(x)/\chi(x)$ and hence $\chi'=(\chi_{gr}\circ \varphi_{Q,\alpha})\chi$.
\end{proof}
In this paragraphe the quandle $Q$ is finite. We have proved in the previous subection that an irreducible quandle representation $\rho$ of $Q$ induces an irreducible projective representation of $\rho_{Inn(Q)}$ of $Inn(Q)$.
Proposition \ref{tc} shows that irreducible projective representations of $Inn(Q)$ that lift to linear representations of $G(Q)_\alpha$ (via $\bar{\theta}_Q$) are exactly those that lift to $G(Q)$ (via $\theta_Q$). This proves that projective irreducible representations of $Inn(Q)$ that are induced by irreducible quandle representations of $Q$ are those that lift to $G(Q)_\alpha$. Now $G(Q)_\alpha$ is a stem extension of $Inn(Q)$, hence its nucleus $q_\alpha(Tor(Z_0))\simeq Tor(Z_0)$ correspond (by the inflation-restriction exact sequence with coefficients in $\mathbb{C}^\times$) to a subgroup $M(Q)_\alpha$ of the Schur multiplier $H^2(Inn(Q),\mathbb{C^\times})$. The classes defined by $M(Q)_\alpha$ are the classes of irreducible projective representations of $Inn(Q)$ that lift to linear represenations of $G(Q)_\alpha$ and hence these are the classes of irreducible representations induced by irreducible representations of $Q$ since they lift to $G(Q)$ also. In particular, $M(Q)_\alpha\simeq Tor(Z_0)$ does not depend on $\alpha$ and allows to determine wich projective irreducible representations of $Inn(Q)$ are induced by quandle irreducible representations. Hence, we have the following :

\begin{theorem}
A finite quandle $Q$ defines a subgroup $M_Q$ of $H^2(Inn(Q),\mathbb{C}^\times)$. This subgroup is isomorphic to $Tor(Z_0)$ and correpond to the classes of irreducible projective representations of $Inn(Q)$ that are induced by irreducible quandle representations of $Q$. 
\end{theorem}
\begin{corollary}\label{cor}
If $H^2(Inn(Q),\mathbb{C}^\times)=0$ then $Tor(Z_0)=0$.
\end{corollary}
\begin{corollary}
If $H^2(Inn(Q),\mathbb{C}^\times)=0$ then $G(Q)_\alpha$ is isomorphic to $Inn(Q)$.
\end{corollary}
\begin{proposition}
Let $G$ be a Schur Cover. For $Q=Conj(G)$, $Tor(Z_0)\simeq M_Q$ is in fact the torsion group of the center of $G(Q)$ and is isomorphic to $H^2(Inn(G),\mathbb{C}^{\times})$.
\end{proposition}
\begin{proof}
The isomorphism $Tor(Z_0) \simeq M_Q$ is in the previous theorem. The fact that $Tor(Z_0)$ is the torsion of the center of $G(Q)$ follows from the fact that $Conj(G)$ is an injective quandle and that for injective quandle $Z_0$ is all the center (proposition \ref{centinj}). $M_Q$ is isomorphic to the Schur multiplier of $Inn(Q)$ by the previous theorem and remark \ref{Schur}. 
\end{proof}

 \subsection{Presentations of representation groups and algorithmic construction of irreducibles}
Recall that the left translation by $x$ an element of a quandle is denoted by $L_x$. In this subsection we fix a quandle $Q=\{x_1,\dots,x_n\}$. 
\begin{proposition}\label{words}
There are words on $n$ letters $w_1,\dots,w_m$ such that we have an isomorphism from the finitely presented group $$\langle L_{x_1},\dots, L_{x_n} \:\vert  w_1(L_{x_1},\dots,L_{x_n}),\dots,w_m(L_{x_1},\dots,L_{x_n}) \rangle$$ to $Inn(Q)$ mapping the generator $L_{x_i}$ to the left translation $L_{x_i}$.
\end{proposition}
\begin{proof}
Let $F$ be the free group of the $L_{x_i}$'s for $x_i\in Q$. We have a morphism $F\to Inn(Q), L_{x_i} \mapsto L_{x_i}$. Let $R$ be the kernel of this morphism. $R$ is a subgoup of finite index of a free group of finite rank and hence is finitely generated. We take the words to be generators of $R$.
\end{proof}
As in the previous section we will denote by $Z_0$ the kernel of $\theta_Q : G(Q) \to Inn(Q)$ given by $e_x\mapsto L_x$.
\begin{proposition}
$Z_0$ is generated by the elements $w_1(e_{x_1},\dots,e_{x_n}),\dots,w_m(e_{x_1},\dots,e_{x_n})$.
\end{proposition}
\begin{proof}
Let $w(e_{x_1},\dots,e_{x_n})$ be a word lying in $Z_0$. The image of this word by $\theta_Q$ is the word $w(L_{x_1},\dots, L_{x_n})$ this word is zero in $Inn(Q)$ if and only if it can be written using the words $w_1,\dots,w_m$ with entries in $L_{x_i}$. This proves that $w(e_{x_1},\dots,e_{x_n})$ can be writen using the words form the proposition. Obviously the words in the proposition are mapped to zero.
\end{proof}
Recall that the abelianisation of $G(Q)$ is isomorphic to the free abelian group $Z_Q$ on the set of orbits of $Q$ and that the abelianisation map identifies throught this identification to the morphism mapping $e_x$ to the orbit of $x$ for $x\in Q$. Denote by $v_i$ the image of $w_i(e_{x_1},\dots,e_{x_n})$ in $Z_Q$ for $i\in \{1,..,n\}$ by the morphism mapping $e_{x_j}$ to its orbit. Let $V$ be the subgroup of $Z_Q$ generated by the $v_1,\dots,v_m$. Let $b_1,\dots,b_l$ be a basis of $V$. The elements $b_i$ decompose as :
$$b_i=\sum_{j}a_{i,j} v_j,$$
with $a_{i,j}$ integers. The basis and the decompositions can be determined algorithmically using the integer row reduction (Hermite normal form) with a transformation matrix.
\begin{proposition}\label{A}
For $i=1,\dots,l$, set $u_i=w_1(e_{x_1},\dots,e_{x_n})^{a_{i,1}}\dots w_m(e_{x_1},\dots,e_{x_n})^{a_{i,m}}$. Let $A$ be the subgroup of $Z_0$ generated by $u_1,\dots,u_l$.  The group $A$ is a free abelian group and $A\bigoplus Tor(Z_0)$ with $Tor(Z_0)$ the torison of $Z_0$  is equal to $Z_0$. 
\end{proposition}
\begin{proof}
The image of $u_i$ in $Z_Q$ is $b_i$. $A$ is abelian and generated by $l$ elements $u_1,\dots u_l$ its image in $Z_Q$ is the free ablian group on $b_1,\dots,b_l$. It follows that $A$ is free abelian. The image of $A\bigoplus Tor(Z_0)$ in $Z_Q$ is $V$ and $(A\bigoplus Tor(Z_0))\cap G(Q)'$ is $Tor(Z_0)$ since the derived group $G(Q)'$ of $G(Q)$ consist of torsion elements (propsition \ref{abder}). Hence, $A\bigoplus Tor(Z_0)$ is a subgroup of $Z_0$ that maps surjectively on $V$ the image of $Z_0$ in $Z_Q$ and contains the kernel of the morphism $Z_0\to Z_Q$. This proves that $A \bigoplus Tor(Z_0)$ is equal to $Z_0$.
\end{proof}
\begin{corollary}\label{presrep}
$G(Q)/A$ is a representation group with presentation 
$$\langle e_{x_i}, i\in \{1,\dots,n\} \: \vert e_{x_i}e_{x_j}e_{x_i}^{-1}=e_{x_i\triangleright x_j},\: i,j \in \{1,\dots,n\} , u_k=1, \: k\in\{1,\dots,l\} \rangle.$$
where the $u_k$ are as in the previous propostion.
\end{corollary}
The results of this section and theorem \ref{class} combined with the fact that irreducible representations of finite finitely presented groups can be constructed up to equivalence algorithmically gives : 
\begin{theorem}
For a finite quandle $Q$, we can construct algorithmically a list $L$ of irreducible representations of $Q$ such that any irreducible representation of $Q$ is equivalent to a representation in $L$.
\end{theorem}
\begin{proof}
The group $Inn(Q)$ is the subgroup of bijection of $Q$ generated by the left translation. Subgroups of bijections of a finite set given by a set of bijections can be implemented. Hence, we can implement $Inn(Q)$ from the multiplication table of $Q$. The words in proposition \ref{words}, can be constructed for instance by constructing Schreier generators for $R$ in the proof.  Elements $b_i \in V \subset Z_Q$ and the decomposition of $b_i$ in terms of the $v_j$'s as in the paragraphe before proposition \ref{A}, as mentioned before, can be determined algorithmically using the integer row reduction (Hermite normal form) with a transformation matrix. Hence, we can compute a finite presentation (proposition \ref{A} and corollary \ref{presrep}) of a presentation group $G(Q)/A$ of $Q$ algorithmically. Finitely presented groups can be implemented and irreducible representations over $\C$ of finite groups can be constructed up to equivalence algorithmically. This can be done for instance with the software GAP using the command IrreducibleRepresentations. Hence, we can determine the list $\rho_1,\dots, \rho_r$ of irreducible representations of $G(Q)/A$ up to equivalence algorithmically.  Finally, quandle characters of $Q$ are constant maps on orbits of $Q$ and orbits of the action of $Inn(Q)$ can be determined algorithmically. Hence, we can determine the following list of represenetations of $Q$ algorithmically :
$$L:=\{\chi\cdot (\rho_{i}\circ \varphi_{Q,\alpha}) \}_{\chi \in \hat{Q}, i\in \{1,\dots,r\}},$$
where $\hat{Q}$ is the set of quandle characters of $Q$. This could be seen as a finite list of representations with number of orbits of $Q$ parameters comming from the character. By theorem \ref{class}, this list satisfises the conditions of the theorem we are proving. This completes the proof.
\end{proof}
\begin{remark}
The list for which we describe the construction can have two equivalent representations.
\end{remark}
\begin{remark}
One can construct subgroups of finite symmetric groups generated by some permutations using the software "Groups, Algorithms, and Programming" (GAP). Hence, for $Q$ a finite quandle $Inn(Q)$ can be constructed using GAP. A presentation of $Inn(Q)$ on the generators $L_x$ for $x\in Q$ can be determined using GAP. Hence, the words $w_i$ can be constructed using GAP. Hermite normal form and decomposition of the basis is supported in GAP. We therefore can Implement the finitely presented finite group $G(Q)/A$ in $GAP$. Computing a list of irreducibles up to equivalence of a finite group using GAP is possible. Determing orbits of the action of $Inn(Q)$ on $Q$ is possible using GAP. We can hence construct in GAP a function that take $Q$ as a magma for instance and return a list of $r$ $($$r$ as in the proof of the last theorem$)$ irreducible representations. Each representation has the number of orbits of $Q$ parameters. The parameter come from quandle characters and matrices from irreducibles of $G(Q)/A$. These will represent all the irreducible representations of $Q$ up to equivalence. On can reduce the size of the list using proposition \ref{eeeeee}.
\end{remark}
\section{Bogomolov Multiplier and symmetric cocycles}
In this section, we give the definition of the Bogomolov multiplier $B_0(G)$ of a finite group $G$ and show that $B_0(G)$ is isomorphic to the subgroup of $H^2(G,\C^\times)$ consisting of classes of symmetric $2$-cocycles. 
\begin{definition}
For $G$ a finite group, the Bogomolov multiplier $B_0(G)$ is the subgroup of $H^2(G,\mathbb{Q}/\mathbb{Z})$ consisting of classes of $2$-cocycles that give trivial classes after restriction to any abelian subgroup of $G$. 
\end{definition}
\begin{definition}
For $G$ a finite group, we define $B_\mathbb{C}(G)$ as the subgroup of $H^2(G,\mathbb{C}^\times)$ consisting of $2$-cocycle classes that give trivial classes under restriction to any abelian subgroup of $G$. 
\end{definition}

Let $\phi : \mathbb{Q}/\mathbb{Z}\to \mathbb{C}^*$ be the injective group morphism induced by $x\mapsto e^{2i\pi x}$ for $x\in \mathbb{Q}$. To a $2$-cocycle $\alpha : G \times G \to \mathbb{Q}/\mathbb{Z}$ one can assign the $2$-cocycle with values in $\mathbb{C}^\times$ defined by $\phi \circ \alpha$. This induces a morphism $\phi_*^G : H^2(G, \mathbb{Q}/\mathbb{Z})\to H^2(G,\mathbb{C}^\times)$. 
\begin{proposition}
The morphism $\phi_*^G : H^2(G, \mathbb{Q}/\mathbb{Z})\to H^2(G,\mathbb{C}^\times)$ is an isomorphism which is natural in $G$.
\end{proposition}
\begin{proof}
We first prove that the morphism is surjective. The $2$-cocycles with values in $\phi(\mathbb{Q}/\mathbb{Z})$ are in the image of $\phi_*^G$. We will prove that any class in $H^2(G,\mathbb{C}^\times)$ in represented by a $2$-cocycle with values in $\phi(\mathbb{Q}/\mathbb{Z})$. To a class $x$ in $H^2(G,\mathbb{C}^\times)$ correspond an irreducible projective representation $\rho : G \to PGL_n(\mathbb{C})$. Let $\tilde{\rho}:G\to SL_n(\mathbb{C})$ be a map lifting $\rho$ ($SL_n(\C)\to PGL_n(\C)$ is surjective). We have :
$$\tilde{\rho}(gh)=\alpha(g,h)\tilde{\rho}(g)\tilde{\rho}(h)$$
with $\alpha$ a $2$-cocycle representing the class $x$. Since the lifts are taken in $SL_n(\mathbb{C})$, we get by taking the determinant that $\alpha(g,h)^n=1$. Hence, the $2$-cocyle $\alpha $ is with values in $\phi(\mathbb{Q}/\mathbb{Z})$ and represent $x$. This proves that $\phi_*^G$ is surjective. Since $\C^\times$ and $\mathbb{Q}/\mathbb{Z}$ are divisible, we obtain from the universal coefficient theorem, that : 
$$ H^2(G,\mathbb{C}^\times)\simeq Hom(H_2(G,\mathbb{Z}),\mathbb{C}^\times)\quad \text{and} \quad  H^2(G, \mathbb{Q}/\mathbb{Z})\simeq Hom(H_2(G,\mathbb{Z}),\mathbb{Q}/\mathbb{Z}).$$
Since $H_2(G,\mathbb{Z})$ is finite : $$Hom(H_2(G,\mathbb{Z}),\mathbb{C}^\times)\simeq Hom(H_2(G,\mathbb{Z}), \phi(\mathbb{Q}/\mathbb{Z}))\simeq Hom(H_2(G,\mathbb{Z}),\mathbb{Q}/\mathbb{Z}),$$
and therefore the finite groups $H^2(G,\C^\times)$ and $H^2(G,\mathbb{Q}/\mathbb{Z})$ are isomorphic. This complete the proof. Indeed, we have that $\phi_*^G$ is a surjective morphism between two finite isomorphic groups. The fact that it is natural in $G$ follows from the fact that it is defined on cocycles.
\end{proof}
\begin{proposition}\label{BCB0}
The morphism $\phi_*^G$ restricts to an isomorphism beteween the Bogomolov multiplier $B_0(G)$ and the group $B_\mathbb{C}(G)$.
\end{proposition}
\begin{proof}
This follows from the definitions of both groups and the fact that for an abelian subgroup $A$ of $G$ we have a commutative diagram : 
$$\begin{tikzcd}
H^2(G,\mathbb{C}^\times) \arrow{r} & H^2(A,\mathbb{C}^\times)  \\ 
 H^2(G, \mathbb{Q}/\mathbb{Z})\arrow{u}{\phi_*^G} \arrow{r}&   H^2(A, \mathbb{Q}/\mathbb{Z})\arrow{u}{\phi_*^A}
\end{tikzcd} .$$
where the vertical morphisms are isomorphisms and the horizontal morphisms are the restriction maps. 
\end{proof}

\begin{definition}
A $2$-cocycle $\alpha :G\times G\to \mathbb{C}^\times$ is symmetric if :$$\alpha(g,h)=\alpha(h,g),$$ for all $g,h\in G$. The subgroup of $H^2(G,\C^\times)$ consisting of classes of symmertic $2$-cocycles will be denoted by $H_S^2(G,\mathbb{C}^\times)$.
\end{definition}  
The following result is known. We provide a proof for completness.
\begin{proposition}
If $H$ is finite abelian and $\alpha :H\times H \to \mathbb{C}^\times$ is a symmetric $2$-cocycle, then $\alpha$ is a coboundary.
\end{proposition}
\begin{proof}
Normalizing a symmetric $2$-cocycle gives a cohomologous symmetric $2$-cocycles. Let $\tilde{\alpha}$ be the normalized cocycle of $\alpha$ a symmetric $2$-cocyle. Let $\rho : H \to GL(V)$ be an irreducible $\tilde{\alpha}$-projective representation :
$\rho(gh)=\tilde{\alpha}(g,h)\rho(g)\rho(h).$
Since $\tilde{\alpha}$ is symmetric and $H$ is abelian, considering the above equation for $(g,h)$ and $(h,g)$ gives that $\rho(g)$ and $\rho(h)$ commute. Now let $n$ be the order of $h$. Since $\tilde{\alpha}$ is normalized $\rho(1)=1$ and hence $\rho(h^n)=1=\lambda_h\rho(h)^n$ for some $\lambda_h\in \C^\times$. This proves that $\rho(h)$ is diagonalizable since $x^n-\lambda_h^{-1}$ has simple roots. The image of $\rho$ hence consist of commuting diagonalizable elements.They can be diagonalized in the same basis. It follows that $\rho$ is one dimensional since $\rho$ is irreducible. Since the representation is one dimensional taking the determinant, we get that :
$$\tilde{\alpha}(g,h)=det(\rho(g))^{-1}det(\rho(h))^{-1}det(\rho(gh)).$$
This proves that $\tilde{\alpha}$ is a coboundary since it is the image by the differential of the $1$-cochain $\beta(g)=det(\rho(g))^{-1}$. It follows that $\alpha$ is also a coboundary. 
\end{proof}
\begin{corollary}\label{inc}
For $G$ finite, the class of a symmetric $2$-cocyle of $G$ with values in $\mathbb{C}^\times$ lies in $B_\mathbb{C}(G)$.
\end{corollary}
 
\begin{lemma}
Let $\alpha$ be a cocycle whose class lie in $B_{\mathbb{C}}(G)$. If $g$ and $h$ commute in $G$, then $\alpha(g,h)=\alpha(h,g)$. 
\end{lemma}
\begin{proof}
The restriction of $\alpha\in B_\C(G)$ to the abelian group generated by $g,h$ is a coboundary $\beta$. A coboundary of an abelian group satisfies $\beta(x,y)=\beta(y,x)$ for any $x,y$ in the abelian group.
\end{proof}
\begin{proposition}
The groups $H_S^2(G,\mathbb{C}^\times)$ and $B_0(G)$ are isomorphic.
\end{proposition}
\begin{proof}
In view of proposition \ref{BCB0}, we only need to prove that $H_S^2(G,\mathbb{C}^\times)$ is isomorphic to $B_\mathbb{C}(G)$. Let $x$ be a class in $B_\mathbb{C}(G)$. It is represented by a normalized $2$-cocycle $\alpha$. Denote by, $\mathbb{C}^\times \times_\alpha G$ the central extension of $G$ by $\mathbb{C}^\times$ associated to $\alpha$. For $x,y\in \mathbb{C}^\times$ and $g,h \in G$ :
$$(x,g)(y,h)=(xy\alpha(g,h),gh) \quad \text{and} \quad (y,h)(x,g)=(xy\alpha(h,g),hg).$$
It follows form the previous lemma and the equations that if $h$ and $g$ commute in $G$ then their lifts in $\mathbb{C}^\times \times_\alpha G$ commute. For each conjugacy class $C$ of $G$ choose a representative $g_C$ and set for $h\in G$ :
$$s(hg_Ch^{-1})=(1,h)(1,g_C)(1,h)^{-1}.$$
$s$ defines a section. Indeed, if $h_1g_Ch_1^{-1}=hg_Ch^{-1}$, we have that $h^{-1}h_1$ lies in the centralizer of $g_C$ and hence $(1,h)^{-1}(1,h_1)$ lies in the centralizer of $(1,g_C)$ by the commutativity of lifts of commuting elements we have seen. Now take $g$ and $h$ in G. We have that $g=kg_Ck^{-1}$ for some $k\in G$ and some class $C$. Hence :
\begin{align*}
s(hgh^{-1})&=s(hkg_Ck^{-1}h^{-1})\\
&=(1,hk)(1,g_C)(1,hk)^{-1}\\
&=(\alpha(h,k)^{-1},1)(1,h)(1,k)(1,g_C)(1,k)^{-1}(1,h)^{-1}(\alpha(h,k)^{-1},1)^{-1}\\
&=(1,h)s(g)(1,h)^{-1}\\
&=s(h)s(g)s(h)^{-1}
\end{align*}
where the last line is due to the fact that $s(h)$ and $(1,h)$ differ by a central element. The cocycle $\alpha$ is cohomologous to the cocycle $\beta$ given by :
$$\beta(g,h)=s(g)s(h)s(gh)^{-1}.$$
Conjugating the last equation by $s(h)$ and using that $s(h)s(gh)s(h)^{-1}=s(hg)$, we get since $\beta(g,h)$ is central that :
$$\beta(g,h)=s(h)s(g)s(hg)^{-1}=\beta(h,g).$$
This proves that the class of $\beta$ lies in $H_S^2(G,\mathbb{C}^\times)$ and hence also the class of $\alpha$ lies in $H_S^2(G,\mathbb{C}^\times)$. This proves the inclusion of $B_{\mathbb{C}}(G)$ in $H_S^2(G,\mathbb{C}^\times)$. The inclusion in the other direction is corollary \ref{inc}.
\end{proof}

\section{Representations and enveloping group of conjugacy quandles}
For $G$ a finite group, we show that irreducible representations of $Conj(G)$ are all product of quandle characters of $Conj(G)$ by irreducible linear group representations of $G$ if and only if the Bogomolov multiplier $B_0(G)$ is trivial. We prove that if $G$ finite is perfect and has trivial Bogomolov multiplier then the enveloping group of $Conj(G)$ is isomorphic to $G\times \Z^{c(G)}$, where $c(G)$ is the number of conjugacy classes of $G$.\\\\
Throughout this section $G$ is a finite group, $A(G)$ will denote the enveloping group of the conjugacy quandle $Conj(G)$. $B_0(G)$ is the Bogomolov multiplier of $G$, $H^2_S(G,\C^\times)$ is the subgroup of the Schur multiplier consisting of classes of symmetric $2$-cocycles. Representations are taken over $\C$. We recall that irreducible representations over $\C$ of a finite quandle as $Conj(G)$ are finite dimensional and that we have a one to one correspondance between irreducibles of $Conj(G)$ and irreducibles of $A(G)$ given by $\rho \mapsto \rho_{A(G)}$.
\subsection{Irreducible representations of conjugacy quandles}
\begin{proposition}\label{projrho}
\item[1)] We have a group morphism $\pi :A(G) \to G$ mapping $e_g$ to $g$ for $g\in G$. The kernel of $\pi$ lies in the center of $A(G)$.
\item[2)] An irreducible quandle representation $\rho : Conj(G) \to GL(V)$ induces an irreducible projective representation $\bar{\rho} : G\to PGL(V)$ given by :
$$\bar{\rho}(g)=\overline{\rho(g)},$$
for $g\in G$ and where the overline is for the class in the projective linear group.
\item[3)] Take two irreducible quandle representations $\rho$, $\rho'$ of $Conj(G)$. If $\bar{\rho}=\bar{\rho}'$ then $\rho=\chi \cdot \rho'$ for some quandle character $\chi$ of $Conj(G)$.
\item[4)] Let $\rho$ be an irreducible representation of $Conj(G)$. If $\bar{\rho}$ lifts to a linear representation $\rho'$ of the group $G$, then $\rho=\chi \cdot \rho'$ for some quandle chracter $\chi$. 
\end{proposition}
\begin{proof}
The identity $G\to G$ is quandle morphism hence there is a unique group morphism $\pi :A(G)\to G$ such that $e_g=g$.  The group $Inn(Conj(G))=Inn(G)$. The left translation $L_g$ correspond to the conjugation by $g$ and hence the map $\theta : G\to Inn(Conj(G)),g\mapsto L_g$ is a group morphism. The composition $\theta \circ \pi :A(G) \to Inn(Conj(G))$ is a group morphism mapping $e_g \to L_g$. This is the morphism $\theta_{Conj(G)}$ and its kernel is central (proposition \ref{cent}). This proves that the kernel of $\pi$ lies in the center of $A(G)$. We prove $2)$ Let $\rho$ be an irreducible representation of $Conj(G)$ and let $\rho_{A(G)}$ be the corresponding representation of $A(G)$. Since the center of $A(G)$ acts by scalar in the representation composing $\rho_{A(G)}$ by the projection to the projective linear group gives an irreducible projective representation $\bar{\rho}_{A(G)}$ containing the center in its kernel. Since the kernel of $\pi$ is central, $\bar{\rho}_{A(G)}$ induces an irreducible projective representation $\bar{\rho}$ of $G$.  We have :
$$\bar{\rho}(g)=\bar{\rho}_{A(G)}(e_g)=\overline{\rho_{A(G)}(e_g)}=\overline{\rho(g)},$$
where the overline is for classes in the projective linear group. This completes the proof of $2)$. The proof of $3)$ is similar to the proof of $3)$ of proposition \ref{rhoInn}. We prove $4)$. $\rho$ and $\rho'$ are both irreducible quandle representations of $Conj(G)$. They induce the same projective representation and hence they differ by a character by $3)$.
\end{proof}
\begin{lemma}
A $2$-cocycle $\alpha: G\times G\to \mathbb{C}^\times$ satisfies the equation :
$$\alpha(g,hg^{-1})\alpha(h,g^{-1})\alpha(1,1)^{-1}\alpha(g,g^{-1})^{-1}=1.$$
if and only if it is symmetric.
\end{lemma}
\begin{proof}
By replacing $h$ by $hg$ in the equation, we get the equivalent equation
$$\alpha(g,h)\alpha(hg,g^{-1})\alpha(1,1)^{-1}\alpha(g,g^{-1})^{-1}=1.$$
By the cocycle condition : 
$$\alpha(hg,g^{-1}) \alpha(h,g)\alpha(h,1)^{-1}\alpha(g,g^{-1})^{-1}=1.$$
Combining both equations we get the equivalent equation :
$$ \alpha(g,h)\alpha(1,1)^{-1}\alpha(h,g)^{-1}\alpha(h,1)=1.$$
But by the cocycle condition, $\alpha(h,1)\alpha(1,1)\alpha(h,1)^{-1}\alpha(h,1)^{-1}=1$ and hence $\alpha(h,1)=\alpha(1,1)$. This simplifies the last equation to :
$$ \alpha(g,h)=\alpha(h,g). $$
We have proved the lemma.
\end{proof}
\begin{lemma}\label{lem}
Let $\theta : G \to PGL(V)$ be an irreducible projective representation of $G$. There is an irreducible quandle representation $\rho:Conj(G)\to GL(V)$ inducing $\theta$ (i.e. $\bar{\rho}=\theta$) if and only if the class of $\theta$ in $H^2(G,\mathbb{C}^\times)$ admits a symmetric representative. 
\end{lemma}
\begin{proof}
Let $\theta$ be a an irreducible projective represenation with class $[\theta]\in H^2(G,\mathbb{C}^\times)$. Any lift (as a map) $\gamma : G \to GL(V)$ of $\theta$ satisfies :
$$\gamma(gh)=\alpha(g,h)\gamma(g)\gamma(h),$$
for $\alpha : G\times G\to \mathbb{C}^\times$ a $2$-cocycle representing $[\theta]$ and any $2$-cocycle representing $[\theta]$ is associated to some lift. A lift $\gamma$ is a quandle representation if and only if $\gamma(ghg^{-1})=\gamma(g)\gamma(h)\gamma(g)^{-1}$. Now in term of the cocycle : 
$$\gamma(ghg^{-1})=\alpha(g,hg^{-1})\alpha(h,g^{-1})\gamma(g)\gamma(h)\gamma(g^{-1}),$$
$$\gamma(1)=\alpha(g,g^{-1})\gamma(g)\gamma(g^{-1}),$$
$$\gamma(1)=\alpha(1,1)^{-1}\mathrm{id}.$$
Combining these equations we get :
$$\gamma(ghg^{-1})=\alpha(g,hg^{-1})\alpha(h,g^{-1})\alpha(g,g^{-1})^{-1}\alpha(1,1)^{-1}\gamma(g)\gamma(h)\gamma(g)^{-1}.$$
This proves that a lift $\gamma$ is a quandle representation if and only if :
$$\alpha(g,hg^{-1})\alpha(h,g^{-1})\alpha(g,g^{-1})^{-1}\alpha(1,1)^{-1}=1,$$
wich is equivalent to $\alpha$ symmetric by the previous lemma. What we have seen proves that $\theta$ lifts to a quandle representation if and only if the class $[\theta]$ admits a symmetric representative. It is not hard to check that the lift is irreducible.
\end{proof} 
We recall that we have proved in the previous section that $B_0(G)$ and $H_S^2(G,\C^\times)$ are isomorphic.
\begin{theorem}\label{thsym}
The irreducible quandle representations of $Conj(G)$ over $\mathbb{C}$ are the representations of the form $\chi \cdot \rho$ for $\chi$ a character of $Conj(G)$ and $\rho$ an irreducible linear group representation of $G$, if and only if the isomorphic groups $H_S^2(G,\mathbb{C}^\times)$ and $B_0(G)$ are trivial.
\end{theorem}
\begin{proof}
For $\chi \cdot \rho$ as in the theorem $\overline{\chi\cdot \rho}=\bar{\rho}$. Hence, if the irreducible representations of $G$ are as in the theorem, the class of the projective representation induced by an irreducible quandle representation is zero. Indeed $\bar{\rho}$ lifts to a linear represetation ($\rho$) and hence $\bar{\rho}$ has trivial class in $H^2(G,\C^\times)$. It follows from the previous lemma, that in this case, every symmetric $2$-cocycle has a trivial class in $H^2(G,\mathbb{C}^\times)$ and hence $H_S^2(G,\C^\times)=0$. Now assume that $H_S^2(G,\C^\times)=0$. It follows from the previous lemma that the only irreducible projective representation that lift to an irreducible quandle representation in fact lift to an irreducible linear group representation of $G$. It follows from $4)$ of proposition \ref{projrho} that any irreducible quandle representation of $G$ is of the form $\chi \cdot \rho$ as in the theorem. It is clear that any representation as in theorem is an irreducible quandle representation.
\end{proof} 
\begin{example}
Simple groups and with exceptions quasi-simple groups have trivial Bogomolov multiplier according to $($\cite{Bog}$)$. For instance $A_n$ is simple for $n\geq 5$ and $SL(n,q)$ is quasi-simple for $(n,q)\neq (2,2),(2,3)$.
\end{example}
\begin{example}
The Bogomolov multiplier of $D_{2n}$, $Q_{4n}$, $S_n$ is trivial.
\end{example}
\begin{example}\label{BogSchur}
It follows for instance from the last theorem and proposition \ref{repSchur}, that the Bogomolov multiplier of a Schur cover is trivial. In fact, this can be proved using classical theory of multipliers also.
\end{example}
\begin{remark}
The Bogomolov multiplier of a group is not always trivial. There are groups of order $64$ with non trivial Bogomolov multiplier.
\end{remark}

\subsection{Enveloping groups of conjugacy quandles}
We will denote by $Z_1$ the kernel of $\pi : A(G)\to G$ given by $e_g\mapsto g$. We have seen in the previous subsection that $Z_1$ is central.
\begin{proposition}\label{Z1}
\item[1)] The intersection of the $Z_1$ with the derived subgroup $A(G)'$ of $A(G)$ is equal to $Tor(Z_1)$, the finite torsion subgroup of $Z_1$.
\item[2)] $Z_1/Tor(Z_1)$ is isomorphic to $\Z^{c(G)}$, where $c(G)$ is the number of conjugacy classes of $G$.
\end{proposition}
\begin{proof}
By proposition \ref{abder}, $A(G)'$ consist of the torsion elements of $A(G)$ and is finite. $1)$ follows. From $1)$ $Z_1/Tor(Z_1)$ injects in $A(G)^{ab}$.  It has finite index in $A(G)^{ab}$ since $Z_1$ has finite index in $A(G)$.  Since $A(G)^{ab}$ is free abelian on the orbits of $Conj(G)$, and the orbits of the conjugacy quandle are the conjugacy classes, $A(G)^{ab}$ has rank $c(G)$. Since $Z_1/Tor(Z_1)$ has finite index in $A(G)^{ab}$ it has also rank $c(G)$. 
\end{proof}
 
\begin{proposition}
We have $Tor(Z_1)$ is isomorphic to the isomorphic groups $H_S^2(G,\mathbb{C}^\times)$ and  $B_0(G)$.
\end{proposition}
\begin{proof}
The isomorphism between $H^2_S(G,\C^\times)$ and $B_0(G)$ was established in the previous section. We have an inflation-restrection exact sequence associated to the exact sequence $1\to Z_1\to A(G) \overset{\pi}{\to} G \to 1$ : 
$$1 \to  Hom(G,\mathbb{C}^\times)\to Hom(A(G),\mathbb{C}^\times)\overset{Res}{\to} Hom(Z_1,\mathbb{C}^\times)\to H^2(G,\mathbb{C}^\times)\overset{Inf_\pi}{\to}H^2(A(G),\mathbb{C}^\times),$$
where $Res$ is the restriction map and $Inf_\pi$ is the inflation map associated to $\pi$. By $1)$ of the previous proposition, $Z_1\cap A(G)'=Tor(Z_1)$. hence we have an exact sequence :
$$Hom(A(G),\C^\times)\overset{Res}{\to}Hom(Z_1,\mathbb{C}^\times) \to Hom(Tor(Z_1),\mathbb{C}^\times)\to 0,$$
 where the last map is the restriction morphism. Therefore $coker(Res)\simeq Hom(Tor(Z_1),\mathbb{C}^\times)\simeq Tor(Z_1)$. This proves that $ker(Inf_\pi) \simeq Tor(Z_1)$. On the other hand any class in $H^2(G,\C^\times)$ is the class of an irreducible projective representation (proposition \ref{classir}). Therefore, by proposition \ref{kerinfrho}, the kernel of $Inf_\pi$ correspond to classes of irreducible projective representation of $G$ that lift via $\pi$ to irreducible linear representations of the group $A(G)$. This means that they are classes of irreducible projective representations induced by irreducible quandle representations of $Conj(G)$. These classes are exactly the classes that has symmetric representatives (by lemma \ref{lem}). Therefore, $ker(Inf_\pi)\simeq H_S^2(G,\mathbb{C}^\times)$. Together with the isomorphism $ker(Inf_\pi) \simeq Tor(Z_1)$, we deduce the proposition. 
\end{proof}
We will denote by $Ab : A(G)\to A(G)^{ab}$ the abelianisation morphism
\begin{proposition}
The following equivalent conditions hold if and only if $H_S^2(G,\mathbb{C}^\times)=0$ $($i.e. $Tor(Z_1)=0$, i.e. $B_0(G)=0$$)$.
\item[1)] The morphism $\pi \times Ab : A(G) \to G \times A(G)^{ab}$ is injective.
\item[2)] The morphism $\pi$ restricts to an isomorphism $A(G)'\simeq G'$, where $G'$ is the derived group of $G$.
\item[3)] $Z_1 \simeq \Z^{c(G)}$.
\end{proposition}
\begin{proof}
The equivalence between, $H_S^2(G,\mathbb{C}^\times)=0$, $Tor(Z_1)=0$ and $B_0(G)=0$ was established in the previous proposition. The kernel of $\pi \times Ab$ is $Z_1\cap A(G)'$ and is equal to $Tor(Z_1)$ by proposition \ref{Z1}. This proves the equivalence with $1)$. The kernel of $A(G)'\to G'$ is also $Z_1\cap A(G)'=Tor(Z_1)$. This prove the equivalence with $2)$. The equivalence with $3)$ follows from $2)$ of proposition \ref{Z1}.
\end{proof}
We note that in \cite{leb}, it is shown that the $3$ conditions in the previous proposition are satisfied by what is called $\bar{C}$-groups. In particular we can deduce that finite $\bar{C}$-groups have trivial Bogomolov multiplier.
\begin{theorem}
Let $G$ be a perfect group with trivial Bogomolov multiplier $($e.g. a non-abelian simple group$)$. The enveloping group of $Conj(G)$ is isomorphic to $G\times \Z^{c(G)}$, where $c(G)$ is the number of conjugacy classes of $G$.
\end{theorem}
\begin{proof}
If $B_0(G)=0$ then the $3$ conditions of the previous proposition are satisfied. $1)$ proves that we have an injective morphism $\pi \times Ab : A(G) \to G \times A(G)^{ab}$. Since we assume $G$ perfect. $2)$ of the previous proposition implies that $G\times \{0\}$ lies in the image of $\pi\times Ab$. It follows that $\pi\times Ab$ is surjective because $Ab$ is surjective. Hence, $A(G)\simeq G\times A(G)^{ab}$. Now $A(G)^{ab}$ is the free abelian group on the orbit of the quandle $Conj(G)$ with repect to $Inn(Conj(G))=Inn(G)$. These are the conjugacy classes of $G$. Therefore, $A(G)^{ab}\simeq \Z^{c(G)}$. This completes the proof of the theorem.
\end{proof}
\begin{example}
If $G$ is a Schur cover of finite perfect group, then $G$ satisfies the hypothesis of the theorem $($see example \ref{BogSchur}$)$. 
\end{example}
\section{Representation category}
In this section, we show that, unlike finite groups, two non isomorphic finite quandles can share equivalent Tannakian categories of finite dimensional representations over $\C$. We give three types of examples.
\subsection{Tannakian categories of quandle representations}
We refer to \cite{DM} for the defintion of a Tannakian category. Let $Q$ be a quandle. Denote by $Rep(Q)$ the category whose objects are finite dimensional quandle representation over $\C$ and morphisms are intertwining operators. $Rep(Q)$ is abelian, it can be equipped with a tensor product using the diagonal $Q\to Q\times Q$, and has a forgetfull tensor functor into $Vec$ the category of finite dimensional vector spaces. As for groups, we can construct duals in $Rep(Q)$. $Rep(Q)$ is hence a neutral Tannakian category. The category $Rep(G(Q))$ of finite dimensional representations over $\C$ of $G(Q)$ is naturally a neutral Tannakian category over $Vec$. We have seen that we have a correspondance $\rho$ to $\rho_{G(Q)}$. An intertwinning operator between two quandle representations $\rho$ and $\rho'$ is an intertwinning operator between $\rho_{G(Q)}$ and $\rho'_{G(Q)}$ and vice versa. We hence have an isomorphism between the category $Rep(Q)$ and $Rep(G(Q))$. This isomorphism is in fact an equivalence of Tannakian categories.
\subsection{Non injective quandles}
Let $Q$ be a quandle, denote by $Q_{gr}$ the image of $Q$ in $G(Q)$. $Q_{gr}$ is subquandle of $Conj(G(Q))$ and $\varphi_Q :Q\to G(Q)$ restricts to a surjective quandle morphism $\beta_Q:Q\to Q_{gr}$.
\begin{proposition}
For $Q$ a quandle, $G(Q)$ is isomorphic to $G(Q_{gr})$.
\end{proposition}
\begin{proof}
The composition $\varphi_{Q_{gr}} \circ \beta_Q:Q\to G(Q_{gr})$ is a quandle morphism with $G(Q_{gr})$ endowed with the conjugacy structure. Therefore there exists a group morphism $f :G(Q)\to G(Q_{gr})$ such that :
$$ f\circ \varphi_Q=\varphi_{Q_{gr}} \circ \beta_Q.$$
Denote by $\iota$ the inclusion $Q_{gr} \to G(Q)$. It is a quandle morphism for $G(Q)$ endowed with the conjugacy structure. Hence, there exists a unique group morphism $g :G(Q_{gr})\to G(Q)$ such that :
$$g\circ \varphi_{Q_{gr}}=\iota.$$
Note that : 
$$\iota \circ \beta_Q=\varphi_Q.$$
Now take $x\in Q$.
$$g\circ f(e_x)=g\circ f \circ \varphi_Q(x)=g\circ \varphi_{Q_{gr}}\circ \beta_Q(x)=\iota\circ\beta_Q(x)=\varphi_Q(x)=e_x.$$
Now take $y\in Q_{gr}$ and $\tilde y$ a preimage by $\beta_Q$ of $y$ in $Q$.
$$f\circ g (e_y)=f\circ g\circ\varphi_{Q_{gr}}(y)=f\circ \iota(y)=f\circ \iota \circ \beta_Q(\tilde y) =f\circ \varphi_Q(\tilde y)=\varphi_{gr}\circ \beta_Q(\tilde y)=\varphi_{Q_{gr}}(y)=e_y.$$
The last two equations prove that $f$ and $g$ are inverse to each other. This proves the proposition.
\end{proof}
\begin{proposition}
Let $Q$ be a finite non injective quandle. $Q$ and $Q_{gr}$ are not isomorphic but $Rep(Q)$ and $Rep(Q_{gr})$ are equivalent Tannakian categories.
\end{proposition}
\begin{proof}
Since $Q$ is not injective $\vert Q_{gr} \vert < \vert Q\vert$. By the previous proposition $G(Q)$ and $G(Q_{gr})$ are isomorphic and hence $Rep(G(Q))$ and $Rep(G(Q_{gr}))$ are equivalent Tannakian categories. The proposition follows since for a quandle $Q'$, $Rep(Q')$ and $Rep(G(Q'))$ are equivalent Tannakian categories.
\end{proof}
For finite groups the Tannakian category of representations determines the group this is not the case for quandle in general. One can wonder if injective quandles can be determined from the Tannakian category of representations. We give a negative answer.
\subsection{Dual quandles} An alternative way to define a quandle is to consider a set $Q$ equipped with bijections $L_x$ for each $x\in Q$ such that :
$$L_x(x)=x \quad \text{and} \quad L_xL_y=L_{L_x(y)}L_x,$$
for all $x,y\in Q$. Now :
$$L_x^{-1}(x)=x,$$
for all $x\in Q$. Inverting the equation $L_xL_y=L_{L_x(y)}L_x$ then using the substitution $z=L_x(y)$, we get :
$$L_x^{-1}L_{L_x(y)}^{-1}=L_y^{-1}L_x^{-1} \quad \text{and} \quad L_x^{-1}L_{z}^{-1}=L_{L_x^{-1}(z)}^{-1}L_x^{-1},$$
for all $x,z \in Q$, since $L_x$ is a bijection. This proves that If $Q$ is a quandle the assignement $x\mapsto L_x^{-1}$ defines a quandle :
\begin{definition}
Let $Q$ be a quandle, the dual quandle $Q^*$ of $Q$ is the set $Q$ equipped with the binary operation $\triangleright^{-1}$ given by :
$$x\triangleright^{-1} y=L_x^{-1}(y), $$
for all $x,y$ in $Q$.
\end{definition}
By using the substitution $y=x\triangleright^{-1} z$. We find that the relations $e_{x\triangleright y}=e_xe_ye_x^{-1}$ in the enveloping group $G(Q)$ can be substituted with the relations :
$$e_{x\triangleright^{-1} z}=e_x^{-1}e_ze_x.$$
Denote by $f_x$ the generators of $G(Q^*)$. The correpondance $e_x\mapsto f_x^{-1}$ maps the last relations to the relation :
$$f_{x\triangleright^{-1} z}^{-1}=f_xf_z^{-1}f_x^{-1} \iff f_{x\triangleright^{-1} z}=f_xf_zf_x^{-1}.$$
We therefore have that :
\begin{proposition}
For $Q$ a quandle, the enveloping group $G(Q)$ is isomorphic to $G(Q^*)$.
\end{proposition}
What we have seen till now on dual quandles appears in \cite{DJ}.
\begin{proposition}
The dual quandle of the Alexander quandle $Al(M,T)$ is the quandle $Al(M,T^{-1})$. 
\end{proposition}
\begin{proof}
We have $x\triangleright y= (1-T)x+Ty$. This is equivalent to $y=(1-T^{-1})x+T^{-1}(x\triangleright y)$. Using the substitution $y=x\triangleright^{-1}z$ we get $x\triangleright^{-1}z=(1-T^{-1})x+T^{-1}z$.
\end{proof}
\begin{example}
Denote by $m_2$ and $m_3$ the multiplication by $2$ and $3$ respectively in $\Z/5\Z$. $m_2$ and $m_3$ are automorphisms and they are inverses to each other. We hence have that $Al(\Z/5\Z,m_2)$ is dual to $Al(\Z/5\Z, m_3)$.
\end{example}
\begin{proposition}[\cite{ak}]
Alexander quandles are injective.
\end{proposition}
For an Alexander quandle $Al(M,T)$, $M$ has a natural $\Z[t,t^{-1}]$-module where $t$ acts via $T$. We will denote by $M_T$ this module. The abelian group $(1-T)M$ is a $\Z[t,t^{-1}]$-submodule.
\begin{theorem}[\cite{nelson}]
Let $Al(M,T)$ and $Al(N,T')$ be finite Alexander quandles. These quandles are isomorphic if and only if $\vert M \vert=\vert N\vert$ and the $\Z[t,t^{-1}]$ modules $(1-T)M$ and $(1-T)N$ are isomorphic.
\end{theorem}
In \cite{nelson}, there is a classification up to isomorphism of the quandles $Al(\Z/n\Z,m_a)$ where $m_a$ is the multiplication by the integer $a$. We only need the following special case.
\begin{proposition}[\cite{nelson}]
Take $a,b<n$ assume $a,a-1,b,b-1$ are coprime to $n$. Denote by $m_a$ and by $m_b$ the multiplications by $a$ and $b$ in $\Z/n\Z$. $Al(\Z/n\Z,m_a)$ and $Al(\Z/n\Z, m_b)$ are isomorphic if and only if $a=b$.
\end{proposition}
\begin{proof}
Since $a-1$ and $b-1$ are coprime to $n$. $(1-m_a)M=(1-m_b)M=M$. By the last theorem $Al(\Z/n\Z,m_a)$ and $Al(\Z/n\Z, m_b)$ are isomorphic if and only if the $\Z[t,t^{-1}]$-modules $M_{m_a}$ and $M_{m_b}$ are isomorphic. This is equivalent to $m_a$ and $m_b$ are conjugate in $Aut(\Z/n\Z)$. But $Aut(\Z/n\Z)$ is abelian. Hence the quandles are isomorphic if and only if $m_a=m_b$. Since $a$ and $b$ are strictly smaller then $n$, $m_a=m_b$ is equivalent to $a=b$.
\end{proof}

\begin{example}
$Al(\Z/5\Z,m_2)$ and $Alt(\Z/5\Z,m_3)$ are not isomorphic.
\end{example}
From what we have seen in this section, $Alt(\Z/5\Z,m_2)$ and $Alt(\Z/5\Z,m_3)$ are dual non isomorphic injective quandles. Since they are dual, their enveloping groups are isomorphic and the Tannakian categories of finite dimensional representations over $\C$ of there enveloping groups are equivalent. Since the Tannakian categories $Rep(Q)$ and $Rep(G(Q))$ are equivalent. We have therefore two non isomorphic injective quandles with equivalent Tannakian categories of finite dimensional representations over $\C$.
\subsection{Conjugacy classes of some groups}
The following is theorem 1.28 of \cite{EM}. We give a proof since we did not find a proof in the reference.
\begin{theorem}[\cite{EM}]
Let $G$ be a perfect group such that $H_2(G,\Z)=0$ or equivalently $H^2(G,\C^\times)=0$. Let $C$ be a conjugacy class generating $G$. Denote $Q$ the corresponding subquandle of $Conj(G)$. The enveloping group $G(Q)$ of $Q$ is isomorphic to $G\times \Z$.
\end{theorem}
\begin{proof}
The equivalence between $H_2(G,\Z)=0$ and $H^2(G,\C^\times)=0$ is due to the universal coefficient theorem since $\C^\times$ is divisible. $C$ generates $G$, hence an element of $G$ can be written as product of elements of $C$. Hence the restriction of the conjugation by an element of $G$ to $C$ is an inner automorphism of $Q$. Since the orbit of a point of $C$ under conjugation by elements of $G$ is $C$ and conjugaction by elements of $G$ is an inner automorphism of $Q$, the quandle $Q$ is has one orbit. This proves that $G(Q)^{ab}\simeq \Z$. Now the inclusion $Q\to G$ is a quandle morphism. Hence, we have a group morphism $\pi :G(Q) \to G$ given by $\pi(e_x)=x$ for $x\in C$. Since $C$ generates $G$, $\pi$ is surjective.  Let $\theta : G \to Inn(Q)$ be the morphism mapping $g$ to the restriction of the conjugation by $g$ to $Q$. We have $\theta\circ \pi(e_x)=L_x$. This proves that $\theta \circ \pi$ is the morphism $\theta_{Q}$ of proposition \ref{cent} an hence its kernel is central. It follows that we have a central extension $1\to Z \to G(Q)\to G \to 1$, where $Z$ is the kernel of $\pi$. Since we suppose that the Schur multiplier of $G$ is trivial, it follows from the inflation-restriction exact sequence that $Coker(Res)=0$, where $Res : Hom(G(Q),\C^\times)\to Hom(Z,\C^\times)$ is the restriction morphism. Now, we also have an exact sequence :
$$Hom(G(Q),\C^\times)\overset{Res}{\to} Hom(Z,\C^\times)\to Hom(Z\cap G(Q)',\C^\times)\to 0.$$
where the morphisms are the restriction morphisms and $G(Q)'$ is the derived subgroup of $G(Q)$. Since $coker(Res)=0$ as we have seen, $Hom(Z\cap G(Q)',\C^\times)=0$ and $$Z\cap G(Q)'=1.$$
Denote by $Ab $ the abelianisation morphism $G(Q)\to G(Q)^{ab}$. Since $Z\cap G(Q)'=1$ the morphism $\pi \times Ab$ is injective. Since $G$ is perfect and $\pi$ is surjective :
$$(\pi\times Ab)(G(Q)')=G\times \{0\}.$$
This proves that $\pi \times Ab$ is surjective since $Ab$ is surjective. Hence, $\pi\times Ab:G(Q)\to G \times G(Q)^{ab}$ is an isomorphism. Finally, we have seen earlier in the proof that $G(Q)^{ab}$ is isomorphic to $\Z$. This proves that $G(Q)$ is isomorphic to $G\times \Z$.
\end{proof}
\begin{proposition}
Let $G$ be a finite non-abelian simple group with trivial Schur multiplier. Let $C_1$ and $C_2$ be non trivial conjugacy classes of $G$. Denote by $Q_{C_1}$ and $Q_{C_2}$ the corresponding subquandles of $Conj(G)$. The quandles $Q_{C_1}$ and $Q_{C_2}$ are injective and the Tannakian categories $Rep(Q_{C_1})$ and $Rep(Q_{C_2})$ are equivalent.
\end{proposition}
\begin{proof}
Subquandles of conjugacy quandles are injective as we have seen in the first section. Now a non trivial conjugacy class of a simple group generates the group since it generates a normal subgroup. Hence by the previous theorem $G(Q_{C_1})\simeq G(Q_{C_2})$. Hence the Tannakian categories $Rep(G(Q_{C_1}))$ and $Rep(G(Q_{C_2}))$ are equivalent. This proves the proposition since $Rep(Q)$ and $Rep(G(Q))$ are equivalent as Tannakian categories. 
\end{proof}
The group $PSL(2,8)$ is simple and has trivial Schur multiplier. It has conjugacy classes with different cardinalities. This provide an example of two injective non isomorphic quandles having equivalent Tannakian categories of representations.

\section*{Declaration of generative AI and AI-assisted technologies in the manuscript preparation process :}
During the preparation of this work the author used Google AI and Claude Sonnet 5.5 to search the literature. The ideas behind subsection 7.3 concerning dual quandles were obtained from Google AI and Claude Sonnet 5.5 independently. Claude Sonnet 5.5 was also used to check the manuscript written by the author. After using these tools, the author reviewed and edited the content obtained from AI tools as needed and takes full responsibility for the content of the published article.

 \end{document}